\documentclass[a4paper,11pt]{amsart}

\usepackage[colorlinks=true, linkcolor=red, linktoc=page]{hyperref}
\usepackage{amssymb}
\usepackage{enumitem}
\usepackage{url}
\usepackage{xcolor}
\usepackage{graphicx}
\usepackage{geometry}
\usepackage{subcaption}
\usepackage{tikz}

\usepackage[T1]{fontenc}
\usepackage[english]{babel}
\usepackage{times}
\usepackage{ucs}
\usepackage[utf8x]{inputenc}

\usepackage{dsfont}
\usepackage[mathscr]{euscript}
\usepackage{stmaryrd}

\makeatletter

\def\section{\@startsection{section}{1}%
\z@{.7\linespacing\@plus\linespacing}{.5\linespacing}%
{\normalfont\bfseries\centering}}

\def\@settitle{\begin{center}%
  \baselineskip14\p@\relax
    \bfseries
    \LARGE\@title
  \end{center}%
}

\def\@setauthors{%
  \begingroup
  \trivlist
  \centering\footnotesize \@topsep30\p@\relax
  \advance\@topsep by -\baselineskip
  \item\relax
  \andify\authors
  \def\\{\protect\linebreak}%
 {\Large\authors}%
  \endtrivlist
  \endgroup
}

\def\maketitle{\par
  \@topnum\z@ 
  \@setcopyright
  \thispagestyle{firstpage}
  \ifx\@empty\shortauthors \let\shortauthors\shorttitle
  \else \andify\shortauthors
  \fi
  \@maketitle@hook
  \begingroup
  \@maketitle
  \toks@\@xp{\shortauthors}\@temptokena\@xp{\shorttitle}%
  \toks4{\def\\{ \ignorespaces}}
  \edef\@tempa{%
    \@nx\markboth{\the\toks4
      \@nx{\the\toks@}}{\the\@temptokena}}%
  \@tempa
  \endgroup
  \c@footnote\z@
  \def\do##1{\let##1\relax}%
  \do\maketitle \do\@maketitle \do\title \do\@xtitle \do\@title
  \do\author \do\@xauthor \do\address \do\@xaddress
  \do\email \do\@xemail \do\curraddr \do\@xcurraddr
  \do\commby \do\@commby
  \do\dedicatory \do\@dedicatory \do\thanks \do\thankses
  \do\keywords \do\@keywords \do\subjclass \do\@subjclass
}
\makeatother

\newtheorem{defi}{Definition}[section]
\newtheorem{lem}[defi]{Lemma}
\newtheorem{ass}[defi]{Assumption}
\newtheorem{prop}[defi]{Proposition}
\newtheorem{theo}[defi]{Theorem}
\newtheorem{cor}[defi]{Corollary}
\newtheorem{rk}[defi]{Remark}

\newcommand{\Exp}{\mathbb{E}}
\newcommand{\Var}{\mathrm{Var}}
\renewcommand{\Pr}{\mathbb{P}}
\newcommand{\R}{\mathbb{R}}
\newcommand{\N}{\mathbb{N}}

\newcommand{\dd}{\mathrm{d}}

\newcommand{\ind}[1]{\mathds{1}_{\{#1\}}}

\newcommand{\supp}{\mathrm{supp}}
\newcommand{\dist}{\mathrm{dist}}

\renewcommand{\hat}{\widehat}
\renewcommand{\tilde}{\widetilde}
\renewcommand{\bar}{\overline}

\renewcommand{\emptyset}{\varnothing}

\newcommand{\ie}{{\em i.e.}~}

\newcommand{\petito}{\mathrm{o}}
\newcommand{\grandO}{\mathrm{O}}

\newcommand{\QI}{\mathrm{QI}}
\newcommand{\NN}{\mathrm{NN}}

\title[Reweighting samples under covariate shift using a Wasserstein distance criterion]{Reweighting samples under covariate shift\\ using a Wasserstein distance criterion}
\thanks{This research work has been carried out under the leadership of the Technological Research Institute SystemX, and therefore granted with public funds within the scope of the French Program ”Investissements d’Avenir”.}
\author{Julien Reygner}
\author{Adrien Touboul}
\address{{\bf Julien Reygner}\newline
{\rm \indent CERMICS, Ecole des Ponts, Marne-la-Vallée, France}}
\email{\href{mailto:julien.reygner@enpc.fr}{julien.reygner@enpc.fr}}
\address{{\bf Adrien Touboul}\newline
{\rm \indent CERMICS, Ecole des Ponts, Marne-la-Vallée, France}
\newline
{\rm \indent IRT SystemX, Paris-Saclay, France }
}
\email{\href{mailto:adrien.touboul@enpc.fr}{adrien.touboul@enpc.fr}}

\begin{document}
   
\begin{abstract}  
Considering two random variables with different laws to which we only have access through finite size i.i.d samples, we address how to reweight the first sample so that its empirical distribution converges towards the true law of the second sample as the size of both samples goes to infinity. We study an optimal reweighting that minimizes the Wasserstein distance between the empirical measures of the two samples, and leads to an expression of the weights in terms of Nearest Neighbors. The consistency and some asymptotic convergence rates in terms of expected Wasserstein distance are derived, and do not need the assumption of absolute continuity of one random variable with respect to the other. These results have some application in Uncertainty Quantification for decoupled estimation and in the bound of the generalization error for the Nearest Neighbor regression under covariate shift. 
\end{abstract}
\maketitle
\baselineskip=18pt

\section{Introduction}   

 
\subsection{Regression under covariate shift}\label{sec:statcov}

This article is dedicated to the study of a method aimed at approximating the law of a random variable
\begin{equation}\label{eq:Y}
  Y = f(X, \Theta),
\end{equation}
where $X \in \R^d$, $\Theta \in \mathbf{\Theta}$ are independent random variables, with respective laws denoted by $\mu_X$ and $\mu_\Theta$, and $f : \R^d \times \mathbf{\Theta} \to \R^e$ is a measurable function. The space $\mathbf{\Theta}$ is only assumed to be measurable. The specificity of the problem at stake is that we assume to be provided with:
\begin{itemize}
  \item a \emph{training} sample $(X'_j,Y'_j)_{j \in \llbracket 1,m \rrbracket}$ of i.i.d observations $Y'_j = f(X'_j,\Theta_j)$ where $\Theta_j$ has law $\mu_\Theta$ and is independent from $X'_j$, but the law $\mu_{X'}$ of $X'_j$ may differ from $\mu_X$;
  \item an \emph{evaluation} sample $(X_i)_{i \in \llbracket 1,n \rrbracket}$ with i.i.d observations distributed according to $\mu_X$.
\end{itemize}
This situation is known as \emph{covariate shift} in the statistical learning literature~\cite{braun1995development,quionero2009dataset}. 

This problem is motivated by the study of decomposition-based uncertainty quantification (UQ) methods in complex industrial systems, as is detailed in Subsection~\ref{sec:appUQ} below. In this context, the overall objective is to approximate a \emph{quantity of interest} of the form
\begin{equation}\label{eq:QI}
  \QI =  \Exp[\phi(Y)]
\end{equation}
for some function $\phi: \R^e \to \R$. Following previous works in this direction~\cite{amaral2012decomposition,amaral2014decomposition,amaral2017optimal}, our estimator of $\QI$ assumes the form
\begin{equation}\label{eq:hatQI}
  \hat{\QI}_{m,n} = \frac{1}{m}\sum_{j=1}^m w_j \phi(Y'_j),
\end{equation}
where the vector of \emph{weights} $\mathbf{w}_m = (w_1, \ldots, w_m)$ is chosen so that the \emph{weighted empirical measure}
\begin{equation*}
  \hat{\mu}_{\mathbf{X}'_m}^{\mathbf{w}_m} := \frac{1}{m} \sum_{j=1}^m w_j \delta_{X'_j}
\end{equation*}
of the training sample $\mathbf{X}'_m = (X'_1, \ldots, X'_m)$ be close, in a sense which will be made precise below, to the empirical measure
\begin{equation*}
  \hat{\mu}_{\mathbf{X}_n} := \frac{1}{n} \sum_{i=1}^n \delta_{X_i}
\end{equation*}
of the evaluation sample $\mathbf{X}_n = (X_1, \ldots, X_n)$. Such a reweighting procedure is a standard approach to the problem of \emph{density ratio estimation} in the statistical learning literature~\cite{sugiyama2012density}, the purpose of which is to estimate the density $\dd\mu_X/\dd\mu_{X'}$ from the samples $\mathbf{X}_n$ and $\mathbf{X'}_m$, without estimating separately the measures $\mu_X$ and $\mu_{X'}$. While the theoretical analysis of such methods almost always requires this density to exist, in the UQ context which motivates the present study it is desirable not to assume that any of the measures $\mu_X$ and $\mu_{X'}$ be absolutely continuous with respect to the other, see in particular Remark~\ref{rk:abscont}.

The first step of our work is thus the computation of \emph{optimal} weights $\mathbf{w}_m$ for the problem
\begin{align}
  &\min_{\mathbf{w}_m} W_q\left(\hat{\mu}_{\mathbf{X}_n}, \hat{\mu}_{\mathbf{X}'_m}^{\mathbf{w}_m}\right),\label{eq:wassopt}\\
  &\forall j \in \llbracket1,m\rrbracket, \quad w_j \geq 0, \qquad \text{and} \quad \sum_{j=1}^m w_j = m,\label{eq:condw}
\end{align}
where $W_q$ denotes the \emph{Wasserstein distance} of order $q \geq 1$ on $\R^d$. The reason for the choice of this distance is that unlike criteria already studied in the density ratio estimation literature, such as moment/kernel matching, $L^2$ distance, Kullback--Leibler divergence (see~\cite{sugiyama2012density} and the references therein), it is not sensitive to absolute continuity conditions and therefore it is well suited to our UQ motivation. On the other hand, this choice makes the problem closely related to the fields of \emph{optimal quantization}~\cite{graf2007foundations} and \emph{Nearest Neighbor} (NN) estimation~\cite{biau2015lectures}. 

More precisely, for any $k \in \llbracket 1, m\rrbracket$, denote by $\hat{\psi}^{(k)}_m(x)$ the $k$-NN estimator of the \emph{regression function}
\begin{equation}\label{def:psi}
  \psi(x) := \Exp\left[\phi(Y)|X=x\right] = \Exp\left[\phi(f(x,\Theta))\right],
\end{equation}
defined from the observation of the training sample, and then consider the Monte Carlo estimator
\begin{equation*}
  \hat{\QI}^{(k)}_{m,n} = \frac{1}{n} \sum_{i=1}^n \hat{\psi}^{(k)}_m(X_i)
\end{equation*}
of $\QI$. Then the vector of weights $\mathbf{w}_m$ which are optimal for~\eqref{eq:wassopt}--\eqref{eq:condw} turns out not to depend on the value of $q$, and the associated estimator $\hat{\QI}_{m,n}$ defined by~\eqref{eq:hatQI} coincides with the $1$-NN estimator $\hat{\QI}^{(1)}_{m,n}$. For this reason, we shall denote by $\mathbf{w}_m^{(1)}$ the vector of optimal weights for~\eqref{eq:wassopt}--\eqref{eq:condw}, and more generally by $\mathbf{w}_m^{(k)}$ the vector of weights induced by the $k$-NN estimator of $\psi$.

The main results of this paper describe the asymptotic behavior, as the respective sizes $m$ and $n$ of the training and evaluation samples grow to infinity, of both the Wasserstein distance $W_q(\hat{\mu}_{\mathbf{X}_n},\hat{\mu}_{\mathbf{X}'_m}^{\mathbf{w}^{(k)}_m})$ and the estimator $\hat{\QI}^{(k)}_{m,n}$ of $\QI$. While taking $k=1$ is optimal for the convergence of $\hat{\mu}_{\mathbf{X}'_m}^{\mathbf{w}^{(k)}_m}$ to $\hat{\mu}_{\mathbf{X}_n}$, one may expect from the theory of NN regression that the estimator $\hat{\QI}^{(k)}_{m,n}$ display better convergence properties if $k$ is chosen to grow to infinity with $m$. Therefore we shall study both regimes $k=1$ and $k=k_m \to +\infty$.

\subsection{Outline of the article} The derivation of the Wasserstein optimal weights $\mathbf{w}_m$ is detailed in Section~\ref{sec:wassnn}, where we also highlight connections between our results and various topics in numerical probability and statistical learning. The asymptotic behavior of $W_q(\hat{\mu}_{\mathbf{X}_n},\hat{\mu}_{\mathbf{X}'_m}^{\mathbf{w}^{(k)}_m})$ and $\hat{\QI}^{(k)}_{m,n}$ are respectively studied in Sections~\ref{sec:convan} and~\ref{sec:discussion}. Applications to decomposition-based UQ and the generalization error for NN regression under covariate shift, as well as numerical illustrations, are presented in Section~\ref{sec:appli}.

\subsection{Notation}

We denote by $\N$  the set of the natural integers including zero and by $\N^{*} = \N \setminus \left\{ 0\right\}$ the set of positive integers. Given two integers $n_1 \leq n_2$, the set of the integers between $n_1$ and $n_2$ is written $\llbracket n_1,n_2 \rrbracket = \left\{n_1, \ldots, n_2 \right\}$. For $x \in \R$, $\lceil x \rceil$ (resp. $\lfloor x \rfloor$) is the unique integer verifying $x \leq \lceil x  \rceil < x+1$ (resp.  $x -1 < \lfloor x  \rfloor \leq x$). For  $(x,y) \in \R^{2}$, we use the join and meet notation $x \wedge y = \min(x,y)$ and $x\vee y = \max(x,y)$. Last, we denote by $(x)_+ := 0 \vee x$ and $(x)_- := 0 \vee (-x)$ the nonnegative and nonpositive parts of $x \in \R$.

We fix a norm $|\cdot|$ on $\R^d$, which need not be the Euclidean norm. The supremum norm of $\phi: \R^{d} \rightarrow \R$ is denoted by $\left\| \phi \right\|_{\infty} = \sup_{x \in \R^{d}}|\phi(x)|$. The distance between a point $x \in \R^d$ and a subset $A \subset \R^d$ is denoted by $\dist(x,A)$. Last, for all $x \in \R^d$ and $r \geq 0$, we denote $B(x,r):=\{x' \in \R^d: |x-x'| \leq r\}$, and recall that the \emph{support} of a probability measure $\nu  \in \mathcal{P}(\R^d)$ is defined by
\begin{equation*}
  \supp(\nu) := \left\{x \in \R^{d}  : \forall r > 0,  \nu(B(x,r)) > 0   \right\}.
\end{equation*}

\section{Wasserstein distance minimization and NN regression} \label{sec:wassnn}

\subsection{Optimal weights for Wasserstein distances} We begin by recalling the definition of the Wasserstein distance. 
\begin{defi}[Wasserstein distance]\label{defi:wass}
Let $\mathcal{P}(\R^{d})$ be the set of probability measures on $ \R^{d}$ and, for any $q \in [1,+\infty)$, let 
\begin{equation*}
  \mathcal{P}_{q}( \R^{d}) = \left\{ \nu \in  \mathcal{P}( \R^{d}): \int_{ \R^{d}} |x|^{q}\dd \nu(x) <+ \infty \right\}.
\end{equation*}
The Wasserstein distance of order $q$ between $\mu$ and $\nu \in\mathcal{P}_{q}( \R^{d}) $
is defined as
\begin{equation*}
  W_{q}(\mu,\nu) = \inf \left\{ \int_{ \R^{d} \times  \R^{d}} |x-x'|^{q} \dd\gamma(x,x'): \gamma \in \Pi(\mu,\nu) \right\}^{1/q},
\end{equation*}
where $\Pi(\mu,\nu)$ is the set of probability measures on $ \R^{d} \times  \R^{d}$ with marginals $\mu$ and $\nu$. 
\end{defi}
We refer to~\cite[Section 6]{villani2008optimal} for a general introduction to Wasserstein distances. 

This definition allows for an explicit resolution of the minimization problem~\eqref{eq:wassopt}--\eqref{eq:condw}, which relies on the notion of \emph{Nearest Neighbor} (NN). For  $x \in \R^d$ and $k \in \llbracket 1, m\rrbracket$, we denote by $\NN^{(k)}_{\mathbf{X}'_m}(x)$ the \emph{$k$-th Nearest Neighbor} ($k$-NN) of $x$ among the sample $\mathbf{X}'_m$, that is to say the $k$-th closest point to $x$ among $X'_1, \ldots, X'_m$ for the norm $|\cdot|$. If there are several such points, we define $\NN^{(k)}_{\mathbf{X}'_m}(x)$ to be the point $X'_j$ with lowest index $j$. We omit the superscript notation $(k)$ when referring to the $1$-NN, \ie
\begin{equation*}
\NN_{\mathbf{X}'_m}(x) = \NN^{(1)}_{\mathbf{X}'_m}(x).
\end{equation*}
In the next statement, for any $i \in \llbracket1, n\rrbracket$ and $l \in \llbracket 1, m\rrbracket$, we denote by $j^{(l)}_i$ the (lowest) index $j$ such that $X'_j = \NN^{(l)}_{\mathbf{X}'_m}(X_i)$.
\begin{prop}[Optimal vector of weights]\label{prop:wopt}
 Let the $k$-NN vector of weights $\mathbf{w}_{m}^{(k)} =  (w^{(k)}_1, \ldots, w^{(k)}_m)$ be defined by, for all $j,k \in \llbracket1,m\rrbracket$,  
 \begin{equation}\label{eq:wk}
    w^{(k)}_j := \frac{m}{kn}\sum_{i=1}^n \sum_{ l=1}^{k} \ind{j = j^{(l)}_i}.
  \end{equation}  
  The vector $\mathbf{w}_m^{(k)}$ satisfies~\eqref{eq:condw} and verifies, for all $q \in [1,+\infty)$,
  \begin{equation}
  \label{eq:upboundknn}
    W_q^q\left(\hat{\mu}_{\mathbf{X}_n},\hat{\mu}_{\mathbf{X}'_m}^{\mathbf{w}^{(k)}_m}\right) \leq \frac{1}{kn}\sum_{i=1}^n  \sum_{l=1}^{k} \left|X_i - \NN^{(l)}_{\mathbf{X}'_m}(X_i)\right|^q.
  \end{equation}
  For $k=1$, the equality is reached 
  \begin{equation}\label{eq:upbound1nn}
    W_q^q\left(\hat{\mu}_{\mathbf{X}_n},\hat{\mu}_{\mathbf{X}'_m}^{\mathbf{w}^{(1)}_m}\right) = \frac{1}{n}\sum_{i=1}^n  \left|X_i - \NN_{\mathbf{X}'_m}(X_i)\right|^q,
  \end{equation}
 and the vector is optimal for~\eqref{eq:wassopt} in the sense that for any $\mathbf{w}_m = (w_1, \ldots, w_m)$ which also satisfies~\eqref{eq:condw}, we have 
  \begin{equation}\label{eq:opt1nn}
    W_q\left(\hat{\mu}_{\mathbf{X}_n},\hat{\mu}_{\mathbf{X}'_m}^{\mathbf{w}^{(1)}_m}\right) \leq W_q\left(\hat{\mu}_{\mathbf{X}_n},\hat{\mu}_{\mathbf{X}'_m}^{\mathbf{w}_m}\right).
  \end{equation}
\end{prop}
In other words, for a given $j \in \llbracket 1, m\rrbracket$, $w^{(k)}_j$ is proportional to the number of points $X_i$ of which $X'_j$ is one of the first $k$ NN. We refer to~\cite{loog2012nearest} for a numerical illustration of the use of the vector of weights $\mathbf{w}^{(1)}_{m}$ in the context of classification under covariate shift.
\begin{proof}
  For a general vector of weights $\mathbf{w}_m=(w_1, \ldots, w_m)$ which satisfies~\eqref{eq:condw}, the Wasserstein distance $W^q_q(\hat{\mu}_{\mathbf{X}_n},\hat{\mu}_{\mathbf{X}'_m}^{\mathbf{w}_m})$ is the solution of the following optimal transport problem
\begin{equation}
\label{eq:generaexp}
\begin{split}
 &\inf_{(\gamma_{i,j})_{(i,j) \in\llbracket1,n \rrbracket \times \llbracket1,m \rrbracket}} \sum_{i=1}^{n} \sum_{j=1}^{m}\gamma_{i,j}|X_{i} -X'_{j}|^{q}, \\ 
 &\forall i \in \llbracket1, n\rrbracket, \quad \sum_{j=1}^{m}\gamma_{i,j} = \frac{1}{n} \quad \text{(marginal condition on $\hat{\mu}_{\mathbf{X}_n}$)},\\
  &\forall j \in \llbracket 1,m\rrbracket, \quad \sum_{i=1}^{n}\gamma_{i,j} = \frac{w_{j}}{m} \quad \text{(marginal condition on $\hat{\mu}_{\mathbf{X}'_m}^{\mathbf{w}_m}$)}, \\
  & \forall (i,j) \in \llbracket 1,n\rrbracket \times \llbracket 1,m\rrbracket, \quad \gamma_{i,j} \geq 0,
 \end{split}
\end{equation} 
where $\gamma_{i,j}$ is the coefficient of the discrete transport plan between $\delta_{X_{i}} $ and $\delta_{X'_{j}}$.  For the $k$-NN vector of weights $\mathbf{w}^{(k)}_{m}$ defined by~\eqref{eq:wk}, the transport plan

\begin{equation*}
  \gamma_{i,j}^{(k)} = \frac{1}{kn}\sum_{ l=1}^{k} \ind{j = j^{(l)}_i}
\end{equation*}
satisfies the two marginal conditions. Reordering the terms in the associated cost gives the upper bound of Equation~\eqref{eq:upboundknn}.

We now prove the equality~\eqref{eq:upbound1nn} and optimality~\eqref{eq:opt1nn} of $\mathbf{w}^{(1)}_{m}$ at the same time. On the one hand, it is clear that for any vector of weights $\mathbf{w}_m=(w_1, \ldots, w_m)$ and any transport plan $(\gamma_{i,j})_{(i,j) \in\llbracket1,n \rrbracket \times \llbracket1,m \rrbracket}$ between $\hat{\mu}_{\mathbf{X}_n}$ and $\hat{\mu}_{\mathbf{X}'_m}^{\mathbf{w}_n}$, we have
\begin{align*}
  \sum_{i=1}^n \sum_{j=1}^m\gamma_{i,j}|X_i -X'_j|^q &\geq \sum_{i=1}^n \sum_{j=1}^m\gamma_{i,j}|X_i -\NN_{\mathbf{X}'_m}(X_i)|^q\\
  & = \frac{1}{n}\sum_{i=1}^n |X_i -\NN_{\mathbf{X}'_m}(X_i)|^q,
\end{align*}
therefore taking the infimum over all transport plans yields
\begin{equation*}
  W_q^q\left(\hat{\mu}_{\mathbf{X}_n},\hat{\mu}_{\mathbf{X}'_m}^{\mathbf{w}_m}\right) \geq \frac{1}{n}\sum_{i=1}^n |X_i -\NN_{\mathbf{X}'_m}(X_i)|^q.
\end{equation*}
On the other hand, taking $\mathbf{w}_m=\mathbf{w}^{(1)}_m$ in the left-hand side and combining this inequality with~\eqref{eq:upboundknn} for $k=1$, we obtain both the equality~\eqref{eq:upbound1nn} and optimality~\eqref{eq:opt1nn}.
\end{proof}

\begin{rk}
  In order to alleviate notation, we now write $\hat{\mu}_{\mathbf{X}'_m}^{(k)}=\hat{\mu}_{\mathbf{X}'_m}^{\mathbf{w}^{(k)}_m}$.
\end{rk}

\subsection{Comments on Proposition~\ref{prop:wopt}} In this subsection, we discuss the relation between the result of Proposition~\ref{prop:wopt} and other fields in numerical probability and statistical learning, as well a generalization of this result to a more general framework.

\subsubsection{Link with optimal quantization} It is clear from Proposition~\ref{prop:wopt} that $\hat{\mu}_{\mathbf{X}'_m}^{(1)}$ is the pushforward of $\hat{\mu}_{\mathbf{X}_n}$ by $\NN_{\mathbf{X}'_m}$, and that this transport map yields an optimal coupling between $\hat{\mu}_{\mathbf{X}_n}$ and $\hat{\mu}_{\mathbf{X}'_m}^{(1)}$ in Definition~\ref{defi:wass}, for any $q \geq 1$. The idea to associate each $X_i$ with $\NN_{\mathbf{X}'_m}(X_i) = X'_{j^{(1)}_i}$ is the basis of the theory of \emph{optimal quantization}~\cite{graf2007foundations,Klo12,Pag15}. In this context, the sample $\mathbf{X}'_m$ plays the role of the \emph{quantization grid}, and $\NN_{\mathbf{X}'_m}$ is known to be the \emph{optimal quantization function}. The right-hand side of~\eqref{eq:upbound1nn} then corresponds to the \emph{$L^q$ mean quantization error} induced by the grid $\mathbf{X}'_m$ for the measure $\hat{\mu}_{\mathbf{X}_n}$.

\subsubsection{Link with geometric inference} When $q=2$, the right-hand side of~\eqref{eq:upboundknn} rewrites
\begin{equation*}
  \int_{x \in \R^d} \mathrm{d}^2_{\hat{\mu}_{\mathbf{X}'_m},k/m}(x)\hat{\mu}_{\mathbf{X}_n}(\dd x),
\end{equation*}
where $\mathrm{d}_{\mu,\alpha}(\cdot)$ is the \emph{distance function to $\mu$ with parameter $\alpha$} introduced by Chazal, Cohen-Steiner and Mérigot in~\cite[Definition~3.2]{ChaCohMer11} in order to perform geometric and topological inference for set estimation, see also~\cite{ChaMasMic16,BreLev20} for robust inference. In particular, \cite[Proposition~3.3]{ChaCohMer11} shows that for any $x \in \R^d$,
\begin{equation}\label{eq:chazal}
  \begin{aligned}
    &\mathrm{d}^2_{\hat{\mu}_{\mathbf{X}'_m},k/m}(x) = \inf_{\mathbf{w}_m} W_2^2\left(\frac{1}{m}\sum_{j=1}^m w_j\delta_{X'_j},\delta_x\right),\\
    & \text{$\mathbf{w}_m = (w_1, \ldots, w_m)$ satisfies~\eqref{eq:condw} and $w_j \leq m/k$ for all $j$.}
  \end{aligned}
\end{equation}
This result may be directly compared with the estimates~\eqref{eq:upboundknn} and~\eqref{eq:upbound1nn}, at least in the case where $\mu_X=\delta_x$. In this case, if $k=1$, then the supplementary constraint $w_j \leq m/k$ is necessarily implied by~\eqref{eq:condw} and therefore, combining~\eqref{eq:chazal} with~\eqref{eq:upbound1nn}, we recover the optimality result~\eqref{eq:opt1nn}. For arbitrary values of $k$, the combination of~\eqref{eq:chazal} with~\eqref{eq:upboundknn} shows that the vector $\mathbf{w}_m^{(k)}$ need not be optimal for~\eqref{eq:opt1nn}, but yields a solution which is lower than any solution with the supplementary constraint that $w_j \leq m/k$.

\subsubsection{A more general problem} Proposition~\ref{prop:wopt} may appear as a specific instance, restricted to empirical measures, of the following problem: given two probability measures $\mu$ and $\nu$ on $\R^d$, and assuming that $\mu \in \mathcal{P}_q(\R^d)$, compute the infimum of $W_q(\mu,\rho\dd\nu)$ over all probability densities $\rho$ with respect to $\nu$. Similar questions were recently addressed in~\cite{ButCarLab20}. First, it is clear that if $\mu$ is absolutely continuous with respect to $\nu$, then taking $\rho = \dd\mu/\dd\nu$ shows that this minimum is $0$. Next, following the proof of Proposition~\ref{prop:wopt}, it is easily seen that for any $\rho$,
\begin{equation*}
  W_q^q(\mu,\rho\dd\nu) \geq \Exp\left[\dist(X,\supp(\nu))^q\right], \qquad X \sim \mu.
\end{equation*}
To show that the right-hand side actually matches with the infimum of the left-hand side when $\rho$ varies, we keep following the proof of Proposition~\ref{prop:wopt}. Since $\supp(\nu)$ is closed, for any $x \in \R^d$ the set $\Psi(x) := \{x' \in \supp(\nu): |x-x'| = \dist(x,\supp(\nu))\}$ is nonempty and closed. Besides, the multifunction $\Psi$ is weakly measurable\footnote{Let us denote $S = \supp(\nu)$ and fix $U \subset \R^d$ an open set, which we write as the countable union of closed sets $(F_n)_{n \geq 1}$. Then $\{x \in \R^d: \Psi(x) \cap U \not= \emptyset \} = \cup_{n \geq 1} \{x \in \R^d: \dist(x,(S \cap F_n)) = \dist(x,S)\}$ and it is easily seen that each set in the right-hand side is measurable.}, therefore by the Kuratowski--Ryll-Nardzewski theorem it admits a measurable selection which we denote by $\mathrm{nn}_\nu$. We denote by $\mu^*$ the associated pushforward of $\mu$ by $\mathrm{nn}_\nu$. Then it is clear that $\supp(\mu^*) \subset \supp(\nu)$ on the one hand, and that
\begin{equation*}
  W_q^q(\mu,\mu^*) \leq \Exp\left[\dist(X,\supp(\nu))^q\right], \qquad X \sim \mu,
\end{equation*}
on the other hand. We finally deduce from the approximation result stated in Lemma~\ref{lem:approx} below that
\begin{equation*}
  \inf_\rho W_q^q(\mu,\rho\dd\nu) = W_q^q(\mu,\mu^*) = \Exp\left[\dist(X,\supp(\nu))^q\right], \qquad X \sim \mu,
\end{equation*}
which thereby generalizes the results of Proposition~\ref{prop:wopt}. Notice that there may not exist a minimizer $\rho$ for this problem as the measure $\mu^*$ need not be absolutely continuous with respect to $\nu$.

\begin{lem}[$W_q$ approximation by absolutely continuous measures]\label{lem:approx}
  Let $\mu^*$ and $\nu$ be two probability measures on $\R^q$ such that $\supp(\mu^*) \subset \supp(\nu)$. For any $\epsilon > 0$, there exists a probability density $\rho$ with respect to $\nu$ such that, for all $q \geq 1$, $W_q(\mu^*,\rho\dd\nu) < \epsilon$.
\end{lem}
\begin{proof}
  Let $\epsilon>0$ and $X \in \R^d$ be a random variable with law $\mu^*$. Almost surely, $X \in \supp(\mu^*) \subset \supp(\nu)$ and therefore $\nu(B(X,\epsilon))>0$, which allows to draw $Y$ with conditional distribution 
  \begin{equation*}
    \dd\nu(y|B(X,\epsilon)):=\ind{y \in B(X,\epsilon)}\frac{\dd\nu(y)}{\nu(B(X,\epsilon))}.
  \end{equation*}

  On the one hand, the random variable $Y$ has density
  \begin{equation*}
    \rho(y) := \int_{x \in \R^d} \ind{y \in B(x,\epsilon)}\frac{\dd\mu^*(x)}{\nu(B(x,\epsilon))}
  \end{equation*}
  with respect to $\nu$, and on the other hand we have $|X-Y| < \epsilon$, almost surely, which ensures that $W_q(\mu^*,\rho\dd\nu) < \epsilon$ for any $q \geq 1$.
\end{proof}

\section{Analysis of the Wasserstein distance \texorpdfstring{$W_q(\hat{\mu}_{\mathbf{X}_n},\hat{\mu}_{\mathbf{X}'_m}^{(k)})$}{}} \label{sec:convan} 

In this section, we study the asymptotic behavior of $\Exp[W^q_q(\hat{\mu}_{\mathbf{X}_n},\hat{\mu}_{\mathbf{X}'_m}^{(k)})]$ when $n, m \to +\infty$. To this aim, we first notice that by Proposition~\ref{prop:wopt}, we have
\begin{equation}\label{eq:ExpWqq-1}
  \begin{aligned}
    \Exp\left[W_q^q\left(\hat{\mu}_{\mathbf{X}_n}, \hat{\mu}_{\mathbf{X}'_m}^{(1)}\right)\right] &= \Exp\left[\frac{1}{n}\sum_{i=1}^n \left|X_i - \NN_{\mathbf{X}'_m}(X_i)\right|^q\right]\\
    & = \Exp\left[\left|X - \NN_{\mathbf{X}'_m}(X)\right|^q\right],
  \end{aligned}
\end{equation}
for $k=1$, and
\begin{equation}\label{eq:ExpWqq-k}
  \begin{aligned}
    \Exp\left[W_q^q\left(\hat{\mu}_{\mathbf{X}_n}, \hat{\mu}_{\mathbf{X}'_m}^{(k)}\right)\right] &\leq \Exp\left[\frac{1}{kn}\sum_{i=1}^n \sum_{l=1}^{k}\left|X_i - \NN^{(l)}_{\mathbf{X}'_m}(X_i)\right|^q\right]\\
    &= \frac{1}{k}\sum_{l=1}^{k} \Exp\left[\left|X - \NN^{(l)}_{\mathbf{X}'_m}(X)\right|^q\right],
  \end{aligned}
\end{equation}
for $k \geq 2$. Observe that the right-hand sides of both~\eqref{eq:ExpWqq-1} and~\eqref{eq:ExpWqq-k} no longer depend on $n$.

\subsection{Consistency} \label{sec:consistency}  Our first main result is a consistency result. Before stating it in Theorem~\ref{theo:consist}, we formulate two crucial assumptions. 

\begin{ass}[Support condition]\label{ass:supp}
  We have $\supp(\mu_X) \subset \supp(\mu_{X'})$.
\end{ass}

\begin{ass}[Min-integrability]\label{ass:integ}
  There exists an integer $m_0 \geq 1$ such that 
  \begin{equation*}
    \Exp\left[\min_{ j \in \llbracket1 ,m_0\rrbracket} |X'_j|\right] < +\infty.
  \end{equation*}
\end{ass}

\begin{theo}[Consistency]\label{theo:consist}
  Let Assumptions~\ref{ass:supp} and~\ref{ass:integ} hold. For all $q \in [1, +\infty)$ such that $\Exp[|X|^q] < +\infty$, and any sequence of positive integers $(k_m)_{m \geq 1}$ such that $k_m / m \to 0$ when $m \to +\infty$, we have
  \begin{equation*}
    \lim_{m \to +\infty} \Exp\left[W_q^q\left(\hat{\mu}_{\mathbf{X}_n}, \hat{\mu}_{\mathbf{X}'_m}^{(k_{m})}\right)\right] = 0,
  \end{equation*}
uniformly in $n$.
\end{theo}
Theorem~\ref{theo:consist} is proved in Subsection~\ref{ss:proofs}.   

\begin{rk}[On Assumption~\ref{ass:integ}]
  Assumption~\ref{ass:integ} is obviously satisfied if $X'$ has a finite first order moment, but also for some heavy-tailed distributions. It writes under the equivalent form
  \begin{equation*}
    \int_0^{+\infty} \Pr(|X'| > r)^{m_0}\dd r < +\infty,
  \end{equation*}
  which may be easier to check. An example of a random variable which does not satisfy this assumption, in dimension $d=1$, is $X'=\exp(1/U)$ where $U$ is a uniform random variable on $[0,1]$.
\end{rk}

\begin{rk}[Limit without support condition] If the support condition of Assumption~\ref{ass:supp} does not hold, then the proof of Theorem~\ref{theo:consist} may be adapted to show that
\begin{equation*}
  \lim_{m,n \to +\infty} \Exp\left[W_q^q\left(\hat{\mu}_{\mathbf{X}_n}, \hat{\mu}_{\mathbf{X}'_m}^{(1)}\right)\right] = \Exp\left[\dist(X,\supp(\mu_{X'}))^q\right].
\end{equation*}

\end{rk}
\subsection{Rates of convergence}
\label{sec:rateconv}
The next step of our study consists in complementing Theorem~\ref{theo:consist} with a rate of convergence. We first discuss the case $k=1$. Following~\eqref{eq:ExpWqq-1}, we start by writing
\begin{equation}\label{eq:nndisteq}
  \begin{split}
    \Exp\left[W_q^q\left(\hat{\mu}_{\mathbf{X}_n}, \hat{\mu}_{\mathbf{X}'_m}^{(1)}\right)\right] &=\Exp\left[\left|X - \NN_{\mathbf{X}'_m}(X)\right|^q\right]\\
    & = \Exp\left[\Exp\left[\left. \left|X - \NN_{\mathbf{X}'_m}(X)\right|^q \right|X\right]\right],
  \end{split}
\end{equation}
and observe that for any $x \in \supp(\mu_X)$, $|x - \NN_{\mathbf{X}'_m}(x)| = \min_{ j \in \llbracket 1 , m \rrbracket} |x-X'_j|$. If there is an open set $U$ of $\R^d$ containing $x$ and such that $\mu_{X'}(\cdot \cap U)$ has a density $p_{X'}$ with respect to the Lebesgue measure which is continuous at $x$, then an elementary computation shows that, for all $r \geq 0$,
\begin{equation*}
  \lim_{m \to +\infty} \Pr\left(m^{1/d} \min_{ j \in \llbracket 1, m \rrbracket} |x-X'_j| > r\right) = \exp\left(-r^d v_d p_{X'}(x)\right),
\end{equation*}
where $v_d$ denotes the volume of the unit sphere of $\R^d$ for the norm $|\cdot|$. If $p_{X'}(x)>0$ then this indicates that the correct order of convergence in Theorem~\ref{theo:consist} should be $m^{-q/d}$. If $p_{X'}(x)=0$, or if the measure $\mu_{X'}(\cdot \cap U)$ is not absolutely continuous with respect to the Lebesgue measure, it is easy to construct elementary examples yielding different rates of convergence; see also~\cite[Chapter~2]{biau2015lectures} for the singular case. We leave these peculiarities apart and work under the following strengthening of the support condition of Assumption~\ref{ass:supp}.

\begin{ass}[Strong support condition]\label{ass:ssupp}
  There exists an open set $U \subset \R^d$ which contains $\supp(\mu_X)$ and such that:
  \begin{enumerate}[label=(\roman*),ref=\roman*]
    \item the measure $\mu_{X'}(\cdot \cap U)$ has a density $p_{X'}$ with respect to the Lebesgue measure;
    \item the density $p_{X'}$ is continuous and positive on $U$;
    \item\label{ass:ssupp:sma} there exist $\kappa \in (0,1]$ and $r_\kappa > 0$ such that, for any $x \in U$, for any $r \in [0,r_\kappa]$,
    \begin{equation*}
      \Pr\left(X' \in B(x,r)\right) \geq \kappa p_{X'}(x) v_d r^d.
    \end{equation*}
  \end{enumerate}
\end{ass}

Obviously, Assumption~\ref{ass:ssupp} implies Assumption~\ref{ass:supp} because then $\supp(\mu_X) \subset U \subset \supp(\mu_{X'})$. Part~\eqref{ass:ssupp:sma} of Assumption~\ref{ass:ssupp} was introduced in~\cite{gadat2016classification} in the context of Nearest Neighbor classification, and called \emph{Strong minimal mass assumption} there. Similar assumptions are commonly used in set estimation, geometric inference and quantization, such as \emph{standardness}~\cite{CueFra97} or \emph{Ahlfors regularity}~\cite{Klo12}. 

Under Assumption~\ref{ass:ssupp}, for all $x \in \supp(\mu_X)$, a positive random variable $Z$ such that $\Pr(Z>r)=\exp(-r^d v_d p_{X'}(x))$ has moments
\begin{equation*}
  \Exp\left[Z^q\right] = \frac{\Gamma(1+q/d)}{(v_d p_{X'}(x))^{q/d}},
\end{equation*}
where $\Gamma$ denotes Euler's Gamma function. Therefore, as soon as the sequence
\begin{equation*}
 m^{q/d} \min_{j \in \llbracket 1, m \rrbracket}\left|X-X'_j \right|^{q}
\end{equation*}
 is uniformly integrable, the normalized quantity
\begin{equation*}
  m^{q/d}\Exp\left[W_q^q\left(\hat{\mu}_{\mathbf{X}_n}, \hat{\mu}_{\mathbf{X}'_m}^{(1)}\right)\right]
\end{equation*}
converges to
\begin{equation*}
  \frac{\Gamma(1+q/d)}{v_d^{q/d}}\Exp\left[\frac{1}{p_{X'}(X)^{q/d}}\right],
\end{equation*}
when  $m$ goes to infinity. This statement appears for example in the literature of stochastic optimal quantization~\cite[Theorem 9.1]{graf2007foundations}. Here, we provide an explicit moment condition ensuring uniform integrability.

\begin{ass}[Moments]\label{ass:mom}
  In addition to Assumption~\ref{ass:ssupp}, the condition
  \begin{equation*}
    \Exp\left[\frac{1+|X|^q}{p_{X'}(X)^{q/d}}\right] < +\infty
  \end{equation*}
  holds.
\end{ass}

Assumptions~\ref{ass:ssupp} and~\ref{ass:mom} are discussed in more detail below. We now state our second main result.

\begin{theo}[Convergence rates for $k=1$]\label{theo:rates}
  Let Assumptions~\ref{ass:integ} and \ref{ass:ssupp} hold, and let $q \in [1,+\infty)$ be such that Assumption~\ref{ass:mom} holds. Then we have
  \begin{equation*}
    \lim_{m \to +\infty} m^{q/d}\Exp\left[W_q^q\left(\hat{\mu}_{\mathbf{X}_n}, \hat{\mu}_{\mathbf{X}'_m}^{(1)}\right)\right] = \frac{\Gamma(1+q/d)}{v_d^{q/d}}\Exp\left[\frac{1}{p_{X'}(X)^{q/d}}\right].
  \end{equation*}
\end{theo}
 
Theorem~\ref{theo:rates} is proved in Subsection~\ref{ss:proofs}. 

We now discuss the estimation of $\hat{\mu}_{\mathbf{X}_n}$ by the weighted empirical measure $\hat{\mu}_{\mathbf{X}'_m}^{(k)}$ for an arbitrary $k \in \llbracket 1, m\rrbracket$. By~\eqref{eq:opt1nn}, we first observe that we always have
\begin{equation*}
  W_q\left(\hat{\mu}_{\mathbf{X}_n}, \hat{\mu}_{\mathbf{X}'_m}^{(k)}\right) \geq W_q\left(\hat{\mu}_{\mathbf{X}_n}, \hat{\mu}_{\mathbf{X}'_m}^{(1)}\right),
\end{equation*}
so that the estimation of $\hat{\mu}_{\mathbf{X}_n}$ is deteriorated by increasing the number of neighbors. Still, in the asymptotic regime of Theorem~\ref{theo:consist}, a bound of the same order of magnitude as Theorem~\ref{theo:rates} may be obtained.
 
\begin{cor}[Convergence rates for $k$-NN]\label{corr:ratesk}Under the assumptions of Theorem~\ref{theo:rates}, for any nondecreasing sequence of positive integers $(k_m)_{m \geq 1}$ such that $k_m / m \to 0$ when $m \to +\infty$, we have
\begin{equation*}
  \underset{m \to +\infty}{\lim \sup} \left(\frac{m}{k_m}\right)^{q/d}\Exp\left[W_q^q\left(\hat{\mu}_{\mathbf{X}_n}, \hat{\mu}_{\mathbf{X}'_m}^{(k_m)}\right)\right] \leq c_{d,q}\frac{\Gamma(1+q/d)}{v_d^{q/d}}\Exp\left[\frac{1}{p_{X'}(X)^{q/d}}\right],
\end{equation*} with some constant $c_{d,q} > 1$.
\end{cor}
Corollary~\ref{corr:ratesk} is proved in Subsection~\ref{ss:proofs}, where the expression of the constant $c_{d,q}$ is also given.
\begin{rk}[Optimal choice of $\mu_{X'}$]
  When $X$ has a density $p_{X}$ with respect to the Lebesgue measure, an interesting fact is that the minimum of the quantity $\Exp\left[1/p_{X'}(X)^{q/d}\right]$ over the probability measure $p_{X'}$ is not reached when $p_{X'} = p_{X}$. Instead, according to  \cite{zador1963development}, the minimum is attained when $p_{X'}(x) \propto p_{X}(x)^{d/(q+d)}$.
\end{rk}

\begin{rk}[NN distance without covariate shift]
  In the case where $\mu_X=\mu_{X'}$, the quantity
  \begin{equation*}
    \Exp\left[W_q^q\left(\hat{\mu}_{\mathbf{X}_n}, \hat{\mu}_{\mathbf{X}'_m}^{(1)}\right)\right] = \Exp\left[\left|X - \NN_{\mathbf{X}'_m}(X)\right|^q\right]
  \end{equation*}
  is called \emph{Nearest Neighbor distance}. It naturally arises in the theoretical study of NN regression and classification~\cite[Chapter 2]{biau2015lectures}. Previous works on the topic focus mainly on the convergence when $q=2$ and assume that $X$ has a bounded support~\cite{biau2015lectures,evans2002asymptotic,liitiainen2008bounds,penrose2011laws}. Some works~\cite{chen2018explaining,kohler2006rates} consider some random variables $X$ with unbounded support in the context of $k$-NN regression, but make the assumption of a bounded regression function $\psi$. 
  
  In this perspective, a direct corollary from Theorem~\ref{theo:rates} is the following statement: let $X$ have a density $p_X$ for which the strong minimal mass assumption~\ref{ass:ssupp}~\eqref{ass:ssupp:sma} holds with $U = \R^d$ and 
  \begin{equation*}
    \int_{\R^{d}} (1+|x|^{q})p_{X}(x)^{1-q/d} \dd x<+\infty.
  \end{equation*}
  Let $\mathbf{X}_n = (X_1, \ldots, X_n)$ be an i.i.d sample from $p_X$, independent from $X$, and let $\NN_{\mathbf{X}_n}$ denote the NN among $\mathbf{X}_n$. We have
  \begin{equation*}
    \Exp\left[|X - \NN_{\mathbf{X}_{n}}(X)|^{q}\right] \underset{n \rightarrow + \infty}{\sim} \frac{\Gamma(1+q/d)}{v_d^{q/d} n^{q/d}}\int_{\R^{d}} p_X(x)^{1-q/d}\dd x.
  \end{equation*}
  This extends the results of the literature by ensuring the asymptotic equivalence for random variables with unbounded support.
\end{rk}

Let us conclude this subsection with some comments on Assumptions~\ref{ass:ssupp} and~\ref{ass:mom}. When $X$ has a compact support, Assumptions \ref{ass:ssupp} and \ref{ass:mom} are verified as soon as $\mu_{X'}$ has a continuous density $p_{X'}$ which is bounded from below and above on an open set $U'$ containing the support of $\mu_{X}$. Indeed, in that case there exist $\epsilon>0$ and an open subset $U$ of $U'$ such that $U$ contains $\supp(\mu_X)$ and $U'$ contains the $\epsilon$-neighborhood of $U$. Then, Assumption~\ref{ass:ssupp}~\eqref{ass:ssupp:sma} is verified with $r_\kappa < \epsilon$ and $\kappa = \inf_{x \in U'} p_{X'}(x)/\sup_{x \in U'} p_{X'}(x)$.

Assumptions~\ref{ass:ssupp} and~\ref{ass:mom} also hold in some nontrivial noncompact cases.  An example of a sufficient condition for Assumption~\ref{ass:ssupp}, which does not depend on $\mu_{X}$, is given in the next statement and is proved in Subsection~\ref{ss:proofs}.
\begin{lem}[Radial density - Sufficient condition for Assumption~\ref{ass:ssupp}]\label{lem:radass} Let $\| \cdot \|$ be a norm on $\R^{d}$, induced by an inner product and not necessarily identical to $|\cdot|$.
  If $\mu_{X'}$ has a density $p_{X'}$ with respect to the Lebesgue measure on $\R^d$, which writes $p_{X'}(x) = h(\|x-x_{0}\|)$ for some $x_0 \in \R^d$ and $h : [0,+\infty) \to \R$ continuous, positive and nonincreasing, then Assumption~\ref{ass:ssupp} holds with $U = \R^{d}$. 
\end{lem}
We also refer to~\cite[Section~2.4]{gadat2016classification} for a discussion of this assumption.

Assumption \ref{ass:mom} gives a relationship between $\mu_{X}$ and $p_{X'}$ to ensure the convergence. In essence, it asserts that the tail of $\mu_{X}$ must be quite lightweight compared to the tail of $p_{X'}$. For instance, if $X$ and $X'$ are centered Gaussian vectors with respective covariance $\sigma^2 I_d$ and $\sigma'^2 I_d$, then by Lemma~\ref{lem:radass}, Assumption~\ref{ass:ssupp} is satisfied with $U=\R^d$, and it is easy to check that for $q \in [1,+\infty)$, Assumption \ref{ass:mom} holds if and only if $\sigma'^2 > \sigma^2 q/d$.

\subsection{Proofs}\label{ss:proofs} In this subsection, we present the proofs of Theorems~\ref{theo:consist} and~\ref{theo:rates}, Corollary~\ref{corr:ratesk} and Lemma~\ref{lem:radass}.

\begin{proof}[Proof of Theorem~\ref{theo:consist}]
We begin our proof with the constant case $k_m = 1$ for all $m$ and then extend it to the general case.
We recall that by~\eqref{eq:ExpWqq-1},
\begin{equation*}
  \Exp\left[W_q^q\left(\hat{\mu}_{\mathbf{X}_n}, \hat{\mu}_{\mathbf{X}'_m}^{(1)}\right)\right] = \Exp\left[\left|X - \NN_{\mathbf{X}'_m}(X)\right|^q\right] = \Exp \left[ \min_{j\in\llbracket 1,m \rrbracket }|X-X'_{j}|^{q} \right].
\end{equation*}

By Assumption \ref{ass:supp}, $X \in\supp(\mu_{X'}) $ almost surely, so that we deduce from Lemma 2.2 in \cite[Chapter 2]{biau2015lectures} that
\begin{equation*}
 \min_{j\in \llbracket 1,m \rrbracket} |X-X'_{j}|^{q} \underset{m \rightarrow +\infty}{\overset{a.s}{\longrightarrow}} 0.
\end{equation*}
Let $m_0$ be the integer given by Assumption~\ref{ass:integ}, we have
\begin{equation*}
\min_{j\in \llbracket 1,m\rrbracket }|X-X'_{j}|^{q}\leq 2^{q-1}\left(|X|^{q} + \min_{j\in\llbracket 1,m\rrbracket}|X'_{j}|^{q}\right).
\end{equation*}
The random variable $|X|^{q} $ is integrable by assumption and for $m \geq \lceil q \rceil m_{0}$, the inequality
\begin{equation*}
\begin{split}
    \Exp\left[ \min_{j \in \llbracket1,m \rrbracket} \left|X^{\prime}_j \right|^{q} \right] &\leq \Exp\left[ \min_{j \in \llbracket1,m \rrbracket} \left|X^{\prime}_j \right|^{\lceil q \rceil} \right]^{q/\lceil q \rceil}\\
    &\leq \Exp\left[ \min_{j \in \llbracket1,m_{0 }\rrbracket} \left|X^{\prime}_j \right|\min_{j \in \llbracket m_{0}+1,2m_{0} \rrbracket} \left|X^{\prime}_j \right|  \cdots \min_{j \in \llbracket (\lceil q \rceil-1)m_{0}+1,\lceil q \rceil m_{0} \rrbracket} \left|X^{\prime}_j \right|\right]^{q/\lceil q \rceil}  \\
    &\leq \Exp \left[\min_{j \in \llbracket1,m_{0 }\rrbracket} \left|X^{\prime}_j \right| \right]^{q} <+\infty
 \end{split}
\end{equation*}
holds. Then by the dominated convergence theorem,
\begin{equation*}
 \Exp \left[ \min_{j \in\llbracket1,m\rrbracket}|X-X'_{j}|^{q} \right] \underset{m \rightarrow + \infty}{\longrightarrow} 0.
\end{equation*} 

For the general case $k_m/m \rightarrow 0$, we adapt directly the proof of \cite[Theorem 2.4]{biau2015lectures} to the context $\mu_X \neq \mu_{X'}$. Let us fix $l \in \llbracket 1,m/2\rrbracket$ and  partition the set $\{X'_1,\ldots, X'_m \}$ into $2l$ sets of size $m_1,\ldots,m_{2l}$ with, for all $j \in \llbracket 1,2l\rrbracket$,
\begin{equation*}
\lfloor m/2l \rfloor \leq m_j  \leq \lfloor  m/2l \rfloor +1.
\end{equation*}
We denote by $\NN^{(1,j)}_{\mathbf{X}'_m}$ the $1$-NN among the subset $j$. By the definition of $\NN^{(l)}_{\mathbf{X}'_{m}}$, there are at least $l$ subsets $j$ for which 
\begin{equation*}
  |X - \NN^{(l)}_{\mathbf{X}'_{m}}(X)| \leq |X-\NN^{(1,j)}_{\mathbf{X}'_m}(X) |,
\end{equation*}
therefore
\begin{equation*}
  |X - \NN^{(l)}_{\mathbf{X}'_{m}}(X) |^{q} \leq \frac{1}{l} \sum_{j=1}^{2l}  |X-\NN^{(1,j)}_{\mathbf{X}'_m}(X) |^{q},
\end{equation*}
and consequently
\begin{equation*}
\Exp\left[\left|X - \NN^{(l)}_{\mathbf{X}'_{m}}(X) \right|^{q} \right] \leq 2\Exp\left[\left| X-\NN_{\mathbf{X}'_{\lfloor m/2l \rfloor}}(X) \right|^{q}\right].
\end{equation*}
Finally, we deduce from~\eqref{eq:ExpWqq-k} that, as soon as $k_m \leq m/2$,
\begin{equation}
\begin{split}
\label{eq:oneboundknn}
   \Exp\left[W_q^q\left(\hat{\mu}_{\mathbf{X}_n}, \hat{\mu}_{\mathbf{X}'_m}^{(k_m)}\right)\right] &\leq \frac{1}{k_{m}}\sum_{l=1}^{k_{m}}\Exp\left[\left|X - \NN^{(l)}_{\mathbf{X}'_{m}}(X) \right|^{q} \right]  \\
   &\leq \frac{2}{k_{m}} \sum_{l=1}^{k_{m}}\Exp\left[\left|X - \NN_{\mathbf{X}'_{\lfloor m/2l\rfloor}}(X) \right|^{q} \right] \\
   & \leq 2\Exp\left[\left|X - \NN_{\mathbf{X}'_{\lfloor m/2k_{m}\rfloor}}(X) \right|^{q} \right], 
\end{split}
\end{equation}
which goes to $0$ as a consequence of the first part of the proof when $m/2k_{m}$ goes to infinity.
\end{proof}

\begin{proof}[Proof of Theorem~\ref{theo:rates}]
By~\eqref{eq:ExpWqq-1}, we have 
\begin{equation}\label{eq:pfrates}
    \begin{split}
   \Exp \left[m^{q/d}W_{q}^{q}(\hat{\mu}_{\mathbf{X}_n},\hat{\mu}_{\mathbf{X}'_m}^{(1)})\right]& =\Exp\left[ \Exp \left[  \left. m^{q/d}\min_{j\in \llbracket 1,m \rrbracket}\left|X-X'_{j}\right|^{q} \right| X  \right] \right] \\
    & = \int_{\R^{d}} \int_{0}^{+\infty}\Pr\left(m^{q/d}\min_{j\in\llbracket1,m\rrbracket}|x-X'_{j}|^{q} >t\right) \dd t\dd\mu_{X}( x) \\
    & =  \int_{\R^{d}} \int_{0}^{+\infty}\Pr(m^{q/d}|x-X'|^{q} >t)^{m}\dd t\dd \mu_{X}( x),
   \end{split}
\end{equation}
by independence of the $X'_{j}$. The proof consists in computing the pointwise limit of $\Pr(m^{p/d}|x-Y|^{p} >t)^{m}$ for $(x,t) \in \supp(\mu_{X})\times \R^{+}$  and then establishing the convergence of the integral via the dominated convergence theorem.

\emph{Pointwise convergence.}
We have
\begin{equation*}
    \begin{split}
        \Pr(m^{q/d}|x-X'|^{q} >t)^{m}& = \left(1 -  \Pr(|x-X'| \leq t^{1/q}/m^{1/d})\right)^{m}\\
        & = \exp\left(m\log\left(1- \Pr(|x-X'| \leq t^{1/q}/m^{1/d})\right)\right).
    \end{split}
\end{equation*}

By Assumption~\ref{ass:ssupp}, we have
\begin{equation*}
 \Pr(|x-X'| \leq t^{1/q}/m^{1/d}) = p_{X'}(x)v_{d}t^{d/q}/m + \petito(1/m),
\end{equation*}
with $v_{d}$ the volume of the unit sphere. Thus 
\begin{equation*}
m\log\left(1- \Pr(|x-X'| \leq t^{1/q}/m^{1/d})\right) =-p_{X'}(x)v_{d}t^{d/q}+ \petito(1),
\end{equation*}
and we conclude that
\begin{equation*}
 \Pr(m^{q/d}|x-X'|^{q} >t)^{m}  \underset{m \rightarrow + \infty}{\longrightarrow} \exp\left(-p_{X'}(x)v_{d}t^{d/q}\right).
\end{equation*}

\emph{Dominated convergence.}
Let $r_{\kappa}>0$ be given by Assumption~\ref{ass:ssupp}. We split the integral in the right-hand side of~\eqref{eq:pfrates} and study each term separately
\begin{equation*} 
\int_{\R^{d}} \int_{0}^{+\infty} \Pr(m^{q/d}|x-X'|^{q} >t)^{m}\dd t\dd \mu_{X}(x) = \mathrm{I} + \mathrm{II}
\end{equation*}
with
\begin{equation*}
\begin{split}
  \mathrm{I}& := \int_{\R^{d}} \int_{0}^{r_{\kappa}^{q}m^{q/d}} \Pr(|x-X'| >t^{1/q}/m^{1/d})^{m}\dd t\dd \mu_{X}(x),  \\
  \mathrm{II}&: =\int_{\R^{d}} \int_{r_{\kappa}^{q}m^{q/d}}^{+\infty} \Pr(|x-X'| >t^{1/q}/m^{1/d})^{m} \dd t\dd \mu_{X}(x).
 \end{split} 
\end{equation*} 

\emph{Convergence of $\mathrm{I}$.} For $t \in [0,r_{\kappa}^{q}m^{q/d}]$, we have $t^{1/q}/m^{1/d} \leq r_{\kappa}$ and thus
\begin{align*}
  \Pr(|x-X'| >t^{1/q}/m^{1/d})^{m}  &= \left(1- \Pr(|x-X'| \leq t^{1/q}/m^{1/d}) \right)^{m}\\
  & \leq \left(1- \frac{p_{X'}(x)v_d\kappa t^{d/q}}{m} \right)^{m} 
\end{align*}
 by Assumption \ref{ass:ssupp}.

Using the elementary inequality $(1-a/n)^n \leq \exp(-a)$ for $a \leq n$, we can write
\begin{equation*}
\begin{split}
 \Pr(|x-X'| >t^{1/q}/m^{1/d})^{m}  & \leq \exp(-\kappa v_d p_{X'}(x) t^{d/q}).
 \end{split}
\end{equation*}
This bound does not depend on $m$ and
the integral
\begin{align*}
  \int_{\R^d}\int_{0}^{+\infty} \exp\left(-\kappa v_d p_{X'}(x) t^{d/q}\right)\dd t \dd\mu_X(x) &= \int_{\R^d}\frac{\Gamma(1+q/d)}{(\kappa v_d p_{X'}(x))^{q/d}} \dd\mu_X(x)\\
  &= \frac{\Gamma(1+q/d)}{(\kappa v_d)^{q/d}} \Exp\left[\frac{1}{p_{X'}(X)^{q/d}}\right]
\end{align*}
 is finite by Assumption~\ref{ass:mom}. We therefore deduce from the dominated convergence theorem that
\begin{align*}
 \lim_{m \to +\infty} \mathrm{I} &= \int_{\R^{d}} \int_{0}^{+\infty}\exp\left(-p_{X'}(x)v_{d}t^{d/q}\right)\dd t \dd\mu_{X}(x)\\
 & = \frac{\Gamma(1+q/d)}{v_{d}^{q/d} }\Exp\left[\frac{1}{p_{X'}(X)^{q/d}}\right] .
\end{align*}

\emph{Convergence of $\mathrm{II}$.} Let $m \geq 2(q+1)m_{0}$. Using the change of variable $r^{q}= t/m^{q/d}$, we have 
\begin{align*}
  \mathrm{II} &= q\int_{\R^{d}} \int_{r_{\kappa}}^{+\infty} m^{q/d}r^{q-1}\Pr(|x-X'| >r)^{m}\dd r\dd \mu_{X} (x)\\
  & \leq q\int_{\R^{d}}  \int_{r_{\kappa}}^{+\infty} V_{m}(x,r)\dd r\dd \mu_{X} (x),
\end{align*}
with
\begin{equation*}
  V_{m}(x,r) := m^{q/d}\Pr(|x-X'| >r_{\kappa})^{m-(q+1)m_{0}}r^{q-1}\Pr(|x-X'| >r)^{(q+1)m_{0}}.
\end{equation*}
As $\Pr(|x-X'| >r_{\kappa})<1$ for all $x$ in $U$, by Assumption~\ref{ass:ssupp}, $V_{m}(x,r)$ is pointwise convergent to $0$ on the support of $\mu_X$. 
We check that $V_m(x,r)$ is bounded from above by an integrable function which does not depend on $m$. Let us denote $m' = m - (q+1)m_0 \geq m/2$ and rewrite
\begin{align*}
  m^{q/d}\Pr(|x-X'| >r_{\kappa})^{m-(q+1)m_{0}} &= \left(\frac{m}{m'}\right)^{q/d} m'^{q/d}\Pr(|x-X'| >r_{\kappa})^{m'}\\
  &\leq 2^{q/d}m'^{q/d}\left(1-\Pr(|x-X'| \leq r_{\kappa})\right)^{m'}\\
  &\leq 2^{q/d}m'^{q/d}\exp\left(-m'\kappa p_{X'}(x)v_dr_{\kappa}^d\right),
\end{align*}
where we have used Assumption~\ref{ass:ssupp} and the elementary above inequality  at the third line. We deduce that
\begin{equation*}
  m^{q/d}\Pr(|x-X'| >r_{\kappa})^{m-(q+1)m_{0}} \leq \frac{C_1}{p_{X'}(x)^{q/d}}, \quad C_1 := \frac{2^{q/d}}{(\kappa v_d r_{\kappa}^d)^{q/d}}\sup_{u \geq 0} (u^{q/d}e^{-u}),
\end{equation*}
so that
\begin{equation}\label{eq:pfrates:2}
  V_m(x,r) \leq \tilde{V}(x,r) := \frac{C_1}{p_{X'}(x)^{q/d}}r^{q-1}\Pr(|x-X'| >r)^{(q+1)m_{0}}.
\end{equation}

To complete the proof, we verify that $\tilde{V}(x,r)$ is integrable on $U\times [r_{\kappa},+\infty)$. We first fix $x \in \R^d$ and estimate the integral of $\tilde{V}(x,r)$ in $r$. Using the fact that if $|x-X'|>r$ then $|X'|>r-|x|$, we first write
\begin{align*}
  \int_{r_{\kappa}}^{+\infty} r^{q-1} \Pr(|x-X'|>r)^{(q+1)m_0}\dd r &\leq \int_{0}^{+\infty} r^{q-1} \Pr(|X'|>r-|x|)^{(q+1)m_0}\dd r\\
  &=\int_{-|x|}^{+\infty} (r+|x|)^{q-1} \Pr(|X'|>r)^{(q+1)m_0}\dd r.
\end{align*}
On the interval $[-|x|,0]$, we have
\begin{equation*}
  \int_{-|x|}^{0} (r+|x|)^{q-1} \Pr(|X'|>r)^{(q+1)m_0}\dd r = \int_{-|x|}^{0} (r+|x|)^{q-1} \dd r = \frac{|x|^q}{q}.
\end{equation*}
On the interval $[0,+\infty)$, we first rewrite
\begin{align*}
  &\int_0^{+\infty} (r+|x|)^{q-1} \Pr(|X'|>r)^{(q+1)m_0}\dd r = \int_{0}^{+\infty} (r+|x|)^{q-1} \Pr\left(\min_{ j \in \llbracket 1 , m_0 \rrbracket} |X'_j|>r\right)^{q+1}\dd r,
\end{align*}
and recall from Assumption~\ref{ass:integ} that $C_2 := \Exp[\min_{j \in \llbracket 1,m_0\rrbracket} |X'_j|]<+\infty$. As a consequence, we deduce from Markov's inequality that the right-hand side in the previous equality is bounded from above by
\begin{equation*}
  \int_{0}^{|x| \vee 1} (r+|x|)^{q-1}\dd r + C_2^{q+1} \int_{|x| \vee 1}^{+\infty} \frac{(r+|x|)^{q-1}}{r^{q+1}}\dd r.
\end{equation*}
If $|x| \leq 1$ then this expression is bounded from above. If $|x|>1$, then we have
\begin{equation*}
  \int_{0}^{|x|} (r+|x|)^{q-1}\dd r \leq 2^{q-1}|x|^q
\end{equation*}
on the one hand, and
\begin{equation*}
  \int_{|x|}^{+\infty} \frac{(r+|x|)^{q-1}}{r^{q+1}}\dd r = \frac{1}{|x|}\int_{1}^{+\infty} \frac{(u+1)^{q-1}}{u^{q+1}}\dd u, 
\end{equation*}
which is bounded from above, on the other hand. Overall, we conclude that there exists a constant $C_3$ such that
\begin{equation}\label{eq:pfrates:3}
  \int_{r_{\kappa}}^{+\infty} r^{q-1} \Pr(|x-X'|>r)^{(q+1)m_0}\dd r \leq C_3(1+|x|^q).
\end{equation}
As a consequence, the combination of~\eqref{eq:pfrates:2} and~\eqref{eq:pfrates:3} yields
\begin{equation*}
  \int_{\R^d} \int_{r_{\kappa}}^{+\infty} \tilde{V}(x,r)\dd r \dd\mu_X(x) \leq C_1C_3 \Exp\left[\frac{1+|X|^q}{p_{X'}(X)^{q/d}}\right],
\end{equation*}
which by Assumption~\ref{ass:mom} allows to apply the dominated convergence theorem to show that $\mathrm{II}$ goes to $0$, and thereby completes the proof.
\end{proof}

\begin{proof}[Proof of Corollary \ref{corr:ratesk}]
We start from the second line of Equation~\eqref{eq:oneboundknn} and estimate its right-hand side
\begin{equation*}
\begin{split}
   \Exp \left[W_{q}^{q}(\hat{\mu}_{\mathbf{X}_n},\hat{\mu}_{\mathbf{X}'_m}^{(k_m)}) \right] &\leq \frac{2}{k_m}\sum_{l=1}^{k_m}\Exp\left[\left|X - \NN_{\mathbf{X}'_{\lfloor m/2l\rfloor}}(X) \right|^{q} \right] \\
    &=\frac{2}{k_m}\sum_{l=1}^{k_m} \left(\frac{2k_m}{m} \frac{l}{k_m} \frac{m}{2l} \right)^{q/d}\Exp\left[\left|X - \NN_{\mathbf{X}'_{\lfloor m/2l\rfloor}}(X) \right|^{q} \right] \\
    &= \left(\frac{k_m}{m} \right)^{q/d}\frac{2^{q/d+1}}{k_m}\sum_{l=1}^{k_m}\left(\frac{l}{k_m} \right)^{q/d}  F\left( \frac{m}{2l}\right)
\end{split}
\end{equation*}
with $F(u) =  u^{q/d}\Exp[|X - \NN_{\mathbf{X}'_{ \lfloor u \rfloor }}(X) |^{q} ]$.
Let $\epsilon > 0$. By Theorem~\ref{theo:rates}, there exists $u_\epsilon \geq 0$ such that, for all $u \geq u_\epsilon$,
\begin{equation*}
\left| F(u)-  \frac{\Gamma(1+q/d)}{v_d^{q/d}}\Exp\left[\frac{1}{p_{X'}(X)^{q/d}}\right] \right| \leq \epsilon.
\end{equation*}
We can remark that for  $m \in \N^{*}$ and $  l \in \llbracket 1, k_m \rrbracket$, 
\begin{equation*}
  \frac{m}{2l} \geq \frac{m}{2k_m } \xrightarrow[m \rightarrow +\infty]{} +\infty.
\end{equation*}
Thus, if we take $m_{\epsilon}$ such that for all $m \geq m_{\epsilon}$,  $ \left\lfloor \frac{m}{2k_m} \right\rfloor \geq u_{\epsilon} $, we have
\begin{equation*}
\left| F\left(\frac{m}{2l} \right)-  \frac{\Gamma(1+q/d)}{v_d^{q/d}}\Exp\left[\frac{1}{p_{X'}(X)^{q/d}}\right] \right| \leq \epsilon
\end{equation*}
for any  $m \geq m_{\epsilon}$ and $l \leq k_m$. Consequently,
\begin{align*}
  &\left |\frac{1}{k_m}\sum_{l=1}^{k_m}\left(\frac{l}{k_m} \right)^{q/d}  \left( F\left(\frac{m}{2l} \right)-  \frac{\Gamma(1+q/d)}{v_d^{q/d}}\Exp\left[\frac{1}{p_{X'}(X)^{q/d}}\right] \right) \right| \leq \epsilon\left| \frac{1}{k_m}\sum_{l=1}^{k_m}\left(\frac{l}{k_m} \right)^{q/d}  \right|\leq \epsilon,
\end{align*}
so that
\begin{equation*}
  \lim_{m \to +\infty} \frac{2^{q/d+1}}{k_m}\sum_{l=1}^{k_m}\left(\frac{l}{k_m} \right)^{q/d}  F\left( \frac{m}{2l}\right) = c_{d,q}\frac{\Gamma(1+q/d)}{v_d^{q/d}}\Exp\left[\frac{1}{p_{X'}(X)^{q/d}}\right] ,
\end{equation*}
where
\begin{align*}
  c_{d,q} &:= \lim_{m \to +\infty} \frac{2^{q/d+1}}{k_m}\sum_{l=1}^{k_m}\left(\frac{l}{k_m} \right)^{q/d}\\
  &= \begin{cases}
    \displaystyle\frac{2^{q/d+1}}{k}\sum_{l=1}^k\left(\frac{l}{k} \right)^{q/d} & \text{if $\sup_{m \geq 1} k_m = k < +\infty$,}\\
    \displaystyle 2^{q/d+1} \int_0^1 u^{q/d}\dd u = \frac{2^{q/d+1}}{q/d+1} & \text{if $\sup_{m \geq 1} k_m = +\infty$,}
  \end{cases}
\end{align*}
because $k_m$ is nondecreasing. This concludes the proof.
\end{proof}

\begin{proof}[Proof of Lemma~\ref{lem:radass}]
Obviously, it suffices to check that $p_{X'}$ satisfies~\eqref{ass:ssupp:sma} in Assumption~\ref{ass:ssupp}. Let us denote by $\langle \cdot,\cdot\rangle$  and $\mathcal{B}(x,r)$  respectively the inner product and  the ball of center $x$  and radius $r$ associated to $\|\cdot\|$.
 We set $x_{0} = 0$ without loss of generality. As $h$ is positive and nonincreasing, we may fix $r_0>0$ and define
\begin{equation*}
  \bar{\kappa} := \frac{h(r_0)}{h(0)} \in (0,1].
\end{equation*}

If $\|x\|\leq r_{0}/2$, then for all $y \in \mathcal{B}(0,r_0/2)$, the monotonicity of $h$ ensures that $p_{X'}(x+y)\geq\bar{\kappa} p_{X'}(x)$. By the equivalence of the norms, there exist $C\geq c>0$ such that for any $x \in \R^{d}$ and any $r\geq0$, $\mathcal{B}(x,cr) \subset B(x,r) \subset \mathcal{B}(x,Cr)  $. Thus
\begin{equation*}
\forall r \leq r_{0}/2c, \qquad \Pr\left(X' \in B(x,r) \right) \geq \Pr\left(X' \in \mathcal{B}(x,cr) \right) \geq (c/C)^{d}v_{d}\bar{\kappa}p_{X'}(x)r^{d}.
\end{equation*}
If $\|x\|> r_{0}/2$, let us introduce the half-cone
\begin{equation*}
 \mathcal{C}_{x} = \left\{x' \in \R^d : \langle x'-x,-x\rangle\geq \frac{\|x'-x\|\|x\|}{2} \right\},
\end{equation*}
and notice that for all $r\leq r_{0}/2$ and $x' \in \mathcal{C}_{x} \cap \mathcal{B}(x,r)$,
\begin{equation*}
    \begin{split}
        \|x'\|^{2}& = \|x\|^{2}+\|x'-x\|^{2} +2 \langle x'-x,x \rangle \\
        & \leq \|x\|^{2}+\|x'-x\|^{2}-\|x'-x\|\|x\| \\
        & \leq  \|x\|^{2}+ \|x'-x\|^{2}-\|x'-x\|^{2} = \|x\|^2.
    \end{split}
\end{equation*}
Thus, for all $x'  \in \mathcal{C}_{x} \cap \mathcal{B}(x,r)$, $p_{X'}(x')\geq p_{X'}(x)$. For a given $r$, the sets $ \mathcal{C}_{x} \cap \mathcal{B}(x,r)$ have the same volume for all $x$, which we denote by $\alpha v_d r^d$ for some $\alpha \in (0,1/C^d)$. Finally, we have 
\begin{equation*}
\forall r \leq r_{0}/2c, \qquad \Pr(X' \in B(x,r))\geq\Pr(X' \in \mathcal{B}(x,cr)\cap \mathcal{C}_{x}) \geq \alpha c^d v_{d}p_{X'}(x)r^{d}.
\end{equation*}
If we take $\kappa = (c/C)^{d} \min(\alpha C^d,\bar{\kappa})$ and $r_{\kappa} = r_{0}/2c$, we obtain the point~\eqref{ass:ssupp:sma} of Assumption~\ref{ass:ssupp}.
\end{proof}

\section{Convergence of \texorpdfstring{$\hat{\QI}_{m,n}^{(k)}$}{} to \texorpdfstring{$\QI$}{}}
\label{sec:discussion}

This section is dedicated to the study of the convergence of $\hat{\QI}_{m,n}^{(k)}$ to $\QI$. As a preliminary step, we complement the results from Section~\ref{sec:convan} by deriving rates of convergence for the Wasserstein distance between $\hat{\mu}_{\mathbf{X}'_m}^{(k)}$ and $\mu_X$ in Subsection~\ref{sec:discconvtheo}. We then distinguish between the \emph{noiseless case} in which $Y=f(X)$, addressed in Subsection~\ref{sec:noiseless}, and the \emph{noisy case} $Y=f(X,\Theta)$, addressed in Subsection~\ref{sec:noisycase}.

\subsection{Convergence of \texorpdfstring{$\hat{\mu}_{\mathbf{X}'_m}^{(k)}$}{} to \texorpdfstring{$\mu_X$}{}}\label{sec:discconvtheo} Let us fix $q \in [1,+\infty)$ and use Jensen's inequality to write, for $k=k_m \in \llbracket 1, m\rrbracket$,
\begin{equation}\label{eq:decomperror}
  \Exp\left[W_q^q\left(\mu_X, \hat{\mu}_{\mathbf{X}'_m}^{(k_m)}\right)\right] \leq 2^{q-1}\left(\Exp\left[W_q^q\left(\mu_X, \hat{\mu}_{\mathbf{X}_n}\right)\right] + \Exp\left[W_q^q\left(\hat{\mu}_{\mathbf{X}_n}, \hat{\mu}_{\mathbf{X}'_m}^{(k_m)}\right)\right]\right).
\end{equation}
Under the assumptions of Corollary~\ref{corr:ratesk}, the second term has order of magnitude at most $(k_m/m)^{q/d}$. The study of the first term, namely the rate of convergence of the expected $W_q$ distance (taken to the power $q$) between the empirical measure of iid realizations and their common distribution, has been the subject of several works. Under the condition that there exists $s>2q$ such that $\Exp[|X|^s] < + \infty$, we have from \cite[Theorem 1]{fournier2015rate}
\begin{equation}\label{eq:fougui}
\Exp\left[ W_q^q\left(\mu_X, \hat{\mu}_{\mathbf{X}_n}\right) \right]= \begin{cases}
  \grandO\left( n^{-1/2} \right) & \text{if $q>d/2$,}\\
  \grandO\left( n^{-1/2}\log(1+n) \right) & \text{if $q=d/2$,}\\
  \grandO\left( n^{-q/d} \right) & \text{if $q<d/2$.}
\end{cases}
\end{equation}
These estimates may be improved if more assumptions are made on $\mu_X$. For example, if this measure possesses a lower and upper bounded density on some bounded subset of $\R^d$, then the rate is known to be $n^{-q/d}$ even if $q>d/2$~\cite{TriSle15}. This rate may even be improved if $\mu_X$ concentrates on a low-dimensional submanifold of $\R^d$~\cite{WeeBac19,Div21}, which is particularly relevant in the UQ context which motivates this study, see Remark~\ref{rk:abscont}. In order to make the use of our results as flexible as possible, from now on we shall denote by $(\tau_{q,d}(n))_{n \geq 1}$ a sequence such that
\begin{equation*}
  \Exp\left[ W_q^q\left(\mu_X, \hat{\mu}_{\mathbf{X}_n}\right) \right] = \grandO\left(\tau_{q,d}(n)\right),
\end{equation*}
and thus
\begin{equation*}
  \Exp\left[W_q^q\left(\mu_X, \hat{\mu}_{\mathbf{X}'_m}^{(k_m)}\right)\right] = \grandO\left(\tau_{q,d}(n)\right) + \grandO\left(\left(\frac{k_m}{m}\right)^{q/d}\right).
\end{equation*}
As is sketched in the discussion above, the precise order of $\tau_{q,d}(n)$ depends on properties of the measure $\mu_X$.

In the sequel, where we study the convergence of $\hat{\QI}_{m,n}^{(k)}$ to $\QI$, the $W_1$ distance plays a specific role, due to the Kantorovitch duality formula \cite[Remark 6.5]{villani2008optimal}
\begin{equation}\label{eq:duality}
 W_{1}\left(\mu, \nu\right) 
  = \sup_{|\varphi|_{\mathrm{Lip}}\leq 1} \left\{ \int_{\R^{d}} \varphi(x) \dd\mu(x) - \int_{\R^{d}} \varphi(x) \dd\nu(x) \right\},
\end{equation}
where $|\varphi|_\mathrm{Lip}$ denotes the Lipschitz constant of $\varphi$.

We shall need the following estimate.

\begin{lem}[$W_1^q$ estimate]\label{lem:W1q}
  If $\Exp[|X|^q] < +\infty$ with $q \geq 2$, then
  \begin{equation*}
    \Exp\left[W_1^q\left(\mu_X, \hat{\mu}_{\mathbf{X}_n}\right) \right] = \grandO\left(\frac{1}{n^{q/2}} + \tau_{1,d}(n)^q\right).
  \end{equation*}
\end{lem}
\begin{proof}
  For any vector $(x_1, \ldots, x_n) \in (\R^d)^n$, let us define
  \begin{align*}
    W(x_1, \ldots, x_n) &:= W_1\left(\frac{1}{n}\sum_{i=1}^n \delta_{x_i}, \mu_X\right)\\
    & = \sup_{|\varphi|_{\mathrm{Lip}}\leq 1} \left\{ \frac{1}{n}\sum_{i=1}^n \varphi(x_i) - \int_{\R^{d}} \varphi(x) \mu_X (\dd x ) \right\}
  \end{align*}
  thanks to~\eqref{eq:duality}. Then for any $i \in \llbracket 1,n\rrbracket$ and $x^\flat_i \in \R^d$, using the identity above and the fact that $|\sup f - \sup g| \leq \sup |f-g|$, we get
  \begin{align*}
    &\left|W(x_1, \ldots, x_n) - W(x_1, \ldots, x_{i-1}, x^\flat_i, x_{i+1}, \ldots, x_n)\right|\\
    &\leq \sup_{|\varphi|_{\mathrm{Lip}}\leq 1} \left| \frac{\varphi(x_1) + \cdots + \varphi(x_n)}{n} - \frac{\varphi(x_1) + \cdots + \varphi(x_i^\flat) + \cdots + \varphi(x_n)}{n}\right|\\
    &= \frac{1}{n}\sup_{|\varphi|_{\mathrm{Lip}}\leq 1}\left|\varphi(x_i) - \varphi(x_i^\flat)\right|\\
    &= \frac{1}{n}|x_i-x^\flat_i|.
  \end{align*} 
  As a consequence, letting $\mathbf{X}_n = (X_1, \ldots, X_n)$ and $\mathbf{X}^\flat_n = (X^\flat_1, \ldots, X^\flat_n)$ be two independent samples from $\mu_X$, we deduce that the random variables $V^+$ and $V^-$ defined by 
  \begin{equation*}
    V^\pm := \Exp\left[\sum_{i=1}^n \left(W(X_1, \ldots, X_n) - W(X_1, \ldots, X_{i-1}, X^\flat_i, X_{i+1}, \ldots, X_n)\right)_\pm^2 | \mathbf{X}_n\right]
  \end{equation*}
  satisfy the bound
  \begin{equation*}
    V^\pm \leq \frac{1}{n^2}\sum_{i=1}^n \Exp\left[|X_i-X^\flat_i|^2 | \mathbf{X}_n\right]. 
  \end{equation*}
  As a consequence, for any $q \geq 2$ we have by Jensen's inequality 
  \begin{align*}
    \Exp\left[(V^\pm)^{q/2}\right] &\leq \frac{1}{n^{q/2}}\Exp\left[\left(\frac{1}{n}\sum_{i=1}^{n}\Exp\left[|X_i-X_i^\flat|^2 | X\right]\right)^{q/2}\right]\\
    & \leq \frac{1}{n^{q/2}} \Exp\left[|X-X^\flat|^q\right] \leq \frac{2^{q-1}}{n^{q/2}} \Exp\left[|X|^q\right].
  \end{align*} 
  
  We therefore deduce from the higher-order Efron--Stein inequality~\cite[Theorem~2]{BouBouLugMas05} that there exists a universal constant $C_q$ such that
  \begin{equation*}
    \Exp\left[\left|W(\mathbf{X}_n)-\Exp\left[W(\mathbf{X}_n)\right]\right|^q\right] \leq C_q \Exp\left[(V^-)^{q/2}+(V^+)^{q/2}\right] \leq \frac{2^qC_q}{n^{q/2}} \Exp\left[|X|^q\right].
  \end{equation*}
  We conclude the proof by writing, using Jensen's inequality again,
  \begin{align*}
    \Exp\left[W(\mathbf{X}_n)^q\right] &= \Exp\left[\left|W(\mathbf{X}_n)-\Exp\left[W(\mathbf{X}_n)\right] + \Exp\left[W(\mathbf{X}_n)\right]\right|^q\right]\\
    &\leq 2^{q-1}\left(\Exp\left[\left|W(\mathbf{X}_n)-\Exp\left[W(\mathbf{X}_n)\right]\right|^q\right] + \Exp\left[W(\mathbf{X}_n)\right]^q\right),
  \end{align*}
  which yields the claimed estimate.
\end{proof}

\begin{rk}
  With the estimates given by~\eqref{eq:fougui}, $\tau_{1,d}(n)$ is of order $n^{-1/d}$ as soon as $d \geq 3$. Thus, in this case, we get from Lemma~\ref{lem:W1q} that $\Exp[W_1^q(\mu_X, \hat{\mu}_{\mathbf{X}_n})]$ is of order $n^{-q/d}$. When $q>d/2$, this is a faster rate of decay to $0$ than what one would have obtained bounding $W_1$ by $W_q$. 
\end{rk}

\subsection{Rate of convergence of \texorpdfstring{$\hat{\QI}_{m,n}^{(k_m)}$}{} in the noiseless case}
\label{sec:noiseless} We assume that $Y = f(X)$ and study the rate of convergence of  $\hat{\QI}_{m,n}^{(k_m)}$ to $\QI$.
When $\phi \circ f $ is  $L$-Lipschitz continuous, we deduce from~\eqref{eq:duality} that
\begin{align*}
  \left| \QI - \hat{\QI}_{m,n}^{(k_m)} \right|^{q} &= \left|\int_{\R^{d}} \phi \circ f(x)  \dd\mu_X(x)  - \int_{\R^{d}} \phi \circ f(x) \dd\hat{\mu}_{\mathbf{X}'_m}^{(k_m)}(x)\right|^{q}\\
  & \leq L^{q}W_{1}^{q}\left(\mu_X,\hat{\mu}_{\mathbf{X}'_m}^{(k_m)}\right).
\end{align*}

We therefore obtain the following result.
\begin{prop}[Rates of convergence in the noiseless case]
\label{prop:ratenoiseless}
Assume that:
\begin{enumerate}[label=(\roman*),ref=\roman*]
  \item the function $f$ does not depend on $\Theta$,
  \item the function $\phi \circ f$ is globally Lipschitz continuous,
\end{enumerate} and let the assumptions of Corollary \ref{corr:ratesk} hold for some $q \geq 2$. Then
\begin{equation}
 \Exp\left[\left| \QI - \hat{\QI}_{m,n}^{(k_m)}  \right|^{q} \right]^{1/q} = \grandO\left(\frac{1}{\sqrt{n}} + \tau_{1,d}(n)\right) + \grandO\left( \left(\frac{k_m}{m} \right)^{1/d}\right).
\end{equation}
\end{prop}
There is no need for $k_m$ to go to infinity and thus $k_m=1$ is optimal.

These computations can be adapted to cases other than $\phi \circ f$ Lipschitz continuous. For instance, if $A \subset \R^{e}$, $\phi(y) = \ind{y \in A}$ and $f$ is globally Lipschitz continuous, it is possible to use the margin assumption of \cite{tsybakov2004optimal} to deduce theoretical rates of convergence in the estimation of $\QI = \Pr(Y \in A)$.

\subsection{Rate of convergence of \texorpdfstring{$\hat{\QI}_{m,n}^{(k_m)}$}{} in the noisy case}
\label{sec:noisycase}We now study the convergence of  $\hat{\QI}_{m,n}^{(k_m)}$ to $\QI$ when $Y = f(X,\Theta)$.
A first striking result is then that even under the assumptions of Theorem~\ref{theo:consist}, the estimator $\hat{\QI}_{m,n}^{(1)}$ need not be consistent. Indeed, consider the case where $X$ is actually deterministic and always equal to some $x_0 \in \R^d$. Then we have
\begin{equation*}
  \hat{\QI}_{m,n}^{(1)} = \frac{1}{n}\sum_{i=1}^n \phi(Y'_{j^{(1)}_i}),
\end{equation*}
where $j^{(1)}_i$ is the index of the closest $X'_j$ to $X_i$. But since $X_i=x_0$ for all $i$, all indices $j^{(1)}_i$ are equal to some $j^{(1)}$ and the estimator rewrites 
\begin{equation*}
  \hat{\QI}_{m,n}^{(1)} = \phi(Y'_{j^{(1)}}) = \phi(f(X'_{j^{(1)}}, \Theta_{j^{(1)}})).
\end{equation*}
While Assumption~\ref{ass:supp} ensures that $X'_{j^{(1)}}$ converges to $x_0$ when $m \to +\infty$, in general the corresponding sequence of $\Theta_{j^{(1)}}$ does not converge.

As is evidenced on this example, the presence of an atom in the law of $X$ makes the estimator $\hat{\QI}_{m,n}^{(1)}$ depend on a single realization of $\Theta$ and therefore prevents this estimator from displaying an averaging behavior with respect to the law of $\Theta$. In Proposition~\ref{prop:consistQI1}, we clarify this point by exhibiting a necessary and sufficient condition for the estimator $\hat{\QI}_{m,n}^{(1)}$ to be consistent, while in Proposition~\ref{prop:consisnoisycase}, we show that replacing $\hat{\QI}_{m,n}^{(1)}$ with $\hat{\QI}_{m,n}^{(k_m)}$ with $k_m \to +\infty$ allows to recover such an averaging behavior and makes the estimator consistent, even when $\mu_X$ has atoms. In the latter case, we also provide rates of convergence in Proposition~\ref{prop:ratenoisy}.

We recall that $\psi(x) = \Exp[\phi(f(x,\Theta))]$ is defined in Equation~\eqref{def:psi}. In the next statement, we denote by $\mathcal{A}_X$ the set of atoms of $\mu_X$, that is to say the set of $x \in \R^d$ such that $\Pr(X=x)>0$, and introduce the notation 
\begin{equation*}
  \vartheta(x) := \Var(\phi(f(x,\Theta))).
\end{equation*}

\begin{prop}[Consistency of the $1$-NN in the noisy case]\label{prop:consistQI1} Assume that:
 \begin{enumerate}[label=(\roman*),ref=\roman*]
   \item the function $\phi$ is bounded,
   \item the function $\psi$ is globally Lipschitz continuous,
   \item the function $\vartheta$ is continuous,
 \end{enumerate}
 and let the assumptions of Theorem~\ref{theo:consist} hold. We have
 \begin{equation*}    
   \Exp \left[ \left|\hat{\QI}_{m,n}^{(1)} -  \QI \right| \right]  \xrightarrow[m,n \rightarrow +\infty]{} 0
 \end{equation*}
 if and only if
 \begin{equation*}
 \forall x \in \mathcal{A}_X, \qquad \vartheta(x)=0.
 \end{equation*}
\end{prop}
In particular, under the above assumptions, if the law of $X$ has no atom, \ie $\mathcal{A}_X = \emptyset$, then $\hat{\QI}_{m,n}^{(1)}$ converges to $\QI$.
\begin{proof}
Let us write
\begin{equation*}
\hat{\QI}_{m,n}^{(1)} - \QI = \left(\hat{\QI}_{m,n}^{(1)} - \tilde{\QI}_{m,n}^{(1)}\right) + \left(\tilde{\QI}_{m,n}^{(1)}- \QI\right),
\end{equation*}
with
\begin{equation*}
\tilde{\QI}_{m,n}^{(1)}=   \frac{1}{m}\sum_{j=1}^m w^{(1)}_j \psi(X'_j).
\end{equation*}
Using the Lipschitz continuity of $\psi$, the duality formula~\eqref{eq:duality} and Theorem~\ref{theo:consist}, we get that $\tilde{\QI}_{m,n}^{(1)}- \QI$ converges to $0$ when $m, n \to +\infty$, in $L^1$. Therefore, $\hat{\QI}_{m,n}^{(1)} - \QI$ converges to $0$ if and only if $\hat{\QI}_{m,n}^{(1)} - \tilde{\QI}_{m,n}^{(1)}$ converges to $0$.

Let us rewrite
\begin{align*}
  \hat{\QI}_{m,n}^{(1)} - \tilde{\QI}_{m,n}^{(1)} &= \frac{1}{m}\sum_{j=1}^m w^{(1)}_j \left(\phi\left(f\left(X'_j,\Theta_j\right)\right)-\psi\left(X'_j\right)\right)\\
  &= \frac{1}{n}\sum_{i=1}^n \left(\phi\left(f\left(X'_{j^{(1)}_i},\Theta_{j^{(1)}_i}\right)\right)-\psi\left(X'_{j^{(1)}_i}\right)\right),
\end{align*} 
introduce the notation
\begin{equation*}
  \mathcal{A}^+_X := \{x \in \mathcal{A}_X : \vartheta(x) > 0\},
\end{equation*}
and denote
\begin{align*}
  e_1 &:= \frac{1}{n}\sum_{i=1}^n \left(\phi(f(X'_{j^{(1)}_i},\Theta_{j^{(1)}_i}))-\psi(X'_{j^{(1)}_i})\right)\ind{X_i \not\in \mathcal{A}^+_X},\\
  e_2 &:= \frac{1}{n}\sum_{i=1}^n \left(\phi(f(X'_{j^{(1)}_i},\Theta_{j^{(1)}_i}))-\psi(X'_{j^{(1)}_i})\right)\ind{X_i \in \mathcal{A}^+_X}.
\end{align*}
In Step~1 below, we prove that
\begin{equation*}
  \Exp\left[\left|e_1 \right|\right] \xrightarrow[n,m \rightarrow +\infty]{} 0,
\end{equation*} 
demonstrating at the same time the direct implication of the convergence when $\mathcal{A}^+_X = \emptyset$. In Step~2, we show that if $\mathcal{A}^+_X\not=\emptyset$ then $\Exp[|e_2|]$ does not converge to $0$, which implies that in this case, $\hat{\QI}_{m,n}^{(1)} - \tilde{\QI}_{m,n}^{(1)}$ does not converge to $0$ in $L^1$. 

In both steps, we shall use the following preliminary remark: given a measurable subset $\mathcal{A}$ of $\R^d$, taking the conditional expectation with respect to $(\mathbf{X}_n,\mathbf{X}'_m)$ it is easy to see that for $i \in \llbracket 1,n\rrbracket$,
\begin{equation*}
  \Exp\left[(\phi(f(X'_{j^{(1)}_i},\Theta_{j^{(1)}_i}))-\psi(X'_{j^{(1)}_i}))\ind{X_i \in \mathcal{A}}\right] = 0,
\end{equation*}
and for $(i_1, i_2) \in \llbracket 1, n\rrbracket^{2}$, 
\begin{align*}
  &\Exp\left[(\phi(f(X'_{j^{(1)}_{i_1}},\Theta_{j^{(1)}_{i_1}}))-\psi(X'_{j^{(1)}_i{}_1}))\ind{X_{i_1} \in \mathcal{A}}(\phi(f(X'_{j^{(1)}_{i_2}},\Theta_{j^{(1)}_{i_2}}))-\psi(X'_{j^{(1)}_{i_2}}))\ind{X_{i_2} \in \mathcal{A}}\right]\\
  & = \Exp\left[\ind{j^{(1)}_{i_1}=j^{(1)}_{i_2}}\vartheta(X'_{j^{(1)}_{i_1}})\ind{X_{i_1} \in \mathcal{A}, X_{i_2} \in \mathcal{A}}\right].
\end{align*}
Therefore,
\begin{align*}
  \Exp[|e_1|^2] &= \frac{1}{n^2}\left(\sum_{i=1}^n\Exp\left[\vartheta(X'_{j^{(1)}_i})\ind{X_i\not\in \mathcal{A}^+_X}\right]+\sum_{i_1 \not= i_2} \Exp\left[\ind{j^{(1)}_{i_1}=j^{(1)}_{i_2}}\vartheta(X'_{j^{(1)}_{i_1}})\ind{X_{i_1} \not\in \mathcal{A}^+_X, X_{i_2} \not\in \mathcal{A}^+_X}\right]\right)\\
  &= \frac{1}{n}\Exp\left[\vartheta(X'_{j^{(1)}_1})\ind{X_1\not \in \mathcal{A}^+_X}\right] + \frac{n-1}{n} \Exp\left[\ind{j^{(1)}_1=j^{(1)}_2}\vartheta(X'_{j^{(1)}_1})\ind{X_1 \not\in \mathcal{A}^+_X, X_2 \not\in \mathcal{A}^+_X}\right],
\end{align*}
and a similar expression holds for $\Exp[|e_2|^2]$.

\emph{Step~1.} Thanks to the boundedness of $\phi$, and thus of $\vartheta$, it is immediate that
\begin{equation*}
  \frac{1}{n}\Exp[\vartheta(X'_{j^{(1)}_1})\ind{X_1\not\in \mathcal{A}^+_X}] \xrightarrow[n \rightarrow +\infty]{} 0,
\end{equation*}
uniformly in $m$. Therefore, to show that $\Exp[|e_1|^2]$ converges to $0$, it suffices to prove that
\begin{equation*}
  \Exp\left[\ind{j^{(1)}_1=j^{(1)}_2}\vartheta(X'_{j^{(1)}_1})\ind{X_1 \not\in \mathcal{A}^+_X, X_2 \not\in \mathcal{A}^+_X}\right] \xrightarrow[m \rightarrow +\infty]{} 0.
\end{equation*}
In this purpose, let us first write
\begin{align*}
  &\Exp\left[\ind{j^{(1)}_1=j^{(1)}_2}\vartheta(X'_{j^{(1)}_1})\ind{X_1 \not\in \mathcal{A}^+_X, X_2 \not\in \mathcal{A}^+_X}\right]\leq \Exp\left[\ind{\NN_{\mathbf{X}'_m}(X_1)=\NN_{\mathbf{X}'_m}(X_2) } \vartheta(\NN_{\mathbf{X}'_m}(X_1))\ind{X_1 \not\in \mathcal{A}^+_X}\right],
\end{align*}
and recall that, by Assumption~\ref{ass:supp} and Lemma 2.2 in \cite[Chapter 2]{biau2015lectures}, $\NN_{\mathbf{X}'_m}(X_1)$ converges to $X_1$ and $\NN_{\mathbf{X}'_m}(X_2)$ converges to $X_2$, almost surely. As a consequence, if $X_1 \in \mathcal{A}_X \setminus \mathcal{A}^+_X$ then $\vartheta(X_1)=0$ and by the continuity of $\vartheta$ and the boundedness of $\phi$, the dominated convergence theorem shows that 
\begin{equation*}
  \Exp\left[\ind{X_1 \in \mathcal{A}_X \setminus \mathcal{A}^+_X}\ind{\NN_{\mathbf{X}'_m}(X_1)=\NN_{\mathbf{X}'_m}(X_2) } \vartheta(\NN_{\mathbf{X}'_m}(X_1))\right] \xrightarrow[m \rightarrow +\infty]{} 0.
\end{equation*}
On the other hand, if $X_1 \not\in \mathcal{A}_X$, then almost surely $X_1\not=X_2$, and therefore $\ind{\NN_{\mathbf{X}'_m}(X_1)=\NN_{\mathbf{X}'_m}(X_2) }$ converges to $0$ almost surely. Using the boundedness of $\phi$ and the dominated convergence theorem again, we deduce that
\begin{equation*}
  \Exp\left[\ind{X_1 \not\in \mathcal{A}_X}\ind{\NN_{\mathbf{X}'_m}(X_1)=\NN_{\mathbf{X}'_m}(X_2) } \vartheta(\NN_{\mathbf{X}'_m}(X_1))\right] \xrightarrow[m \rightarrow +\infty]{} 0,
\end{equation*}
which shows that $\Exp[|e_1|^2]$, and thus $\Exp[|e_1|]$, converge to $0$.

\emph{Step~2.} Let us now assume that $\mathcal{A}^+_X$ is nonempty and show that $e_2$ does not converge to $0$ in $L^1$. We shall actually prove that $e_2$ does not converge to $0$ in $L^2$: since $e_2$ is bounded then this prevents the convergence from occuring in $L^1$. From the preliminary remark, we write
\begin{align*}
  \Exp[|e_2|^2] &= \frac{1}{n}\Exp\left[\vartheta(X'_{j^{(1)}_1})\ind{X_1\in \mathcal{A}^+_X}\right]+ \frac{n-1}{n} \Exp\left[\ind{j^{(1)}_1=j^{(1)}_2}\vartheta(X'_{j^{(1)}_1})\ind{X_1 \in \mathcal{A}^+_X, X_2 \in \mathcal{A}^+_X}\right],
\end{align*}
and we prove that
\begin{equation*}
  \liminf_{m \to +\infty} \Exp\left[\ind{j^{(1)}_1=j^{(1)}_2}\vartheta(X'_{j^{(1)}_1})\ind{X_1 \in \mathcal{A}^+_X, X_2 \in \mathcal{A}^+_X}\right] > 0.
\end{equation*}
Let $x \in \mathcal{A}_X^+$. Obviously,
\begin{align*}
  \Exp\left[\ind{j^{(1)}_1=j^{(1)}_2}\vartheta(X'_{j^{(1)}_1})\ind{X_1 \in \mathcal{A}^+_X, X_2 \in \mathcal{A}^+_X}\right] &\geq \Exp\left[\ind{j^{(1)}_1=j^{(1)}_2}\vartheta(X'_{j^{(1)}_1})\ind{X_1 = X_2 = x}\right]\\
  &= \Exp\left[\vartheta(\NN_{\mathbf{X}'_m}(x))\ind{X_1 = X_2 = x}\right].
\end{align*}
By Assumption~\ref{ass:supp} and Lemma 2.2 in \cite[Chapter 2]{biau2015lectures} again, $\NN_{\mathbf{X}'_m}(x)$ converges to $x$ almost surely, therefore using the continuity and boundedness assumptions on $\vartheta$, the dominated convergence theorem shows that
\begin{equation*}
  \Exp\left[\vartheta(\NN_{\mathbf{X}'_m}(x))\ind{X_1 = X_2 = x}\right] \xrightarrow[m \rightarrow +\infty]{} \vartheta(x)\mu_X(\{x\})^2 > 0,
\end{equation*}
which completes the proof.
\end{proof}

We now study the estimator  $\hat{\QI}_{m,n}^{(k_m)}$ and show that it is unconditionnally consistent as soon as $k_m \rightarrow + \infty$. We provide $L^{2}$ convergence rates in Proposition~\ref{prop:ratenoisy}. 
\begin{prop}[Consistency in the noisy case]
\label{prop:consisnoisycase}
 Assume that 
 \begin{enumerate}[label=(\roman*),ref=\roman*]
   \item the function $\phi$ is bounded,
   \item the function $\psi$ is globally Lipschitz continuous,
 \end{enumerate} 
 and let the assumptions of Theorem~\ref{theo:consist} hold with $q=2$. As soon as $k_m$ goes to infinity with $m$ and $k_m/m \to 0$, we have 
\begin{equation*}
 \Exp\left[\left| \hat{\QI}_{m,n}^{(k_m)} - \QI \right|^{2}\right]^{1/2}  \xrightarrow[n,m \rightarrow + \infty]{} 0.
\end{equation*}
\end{prop}
\begin{proof}
 We decompose the error as
\begin{equation}\label{eq:decompnoisy}
  \hat{\QI}_{m,n}^{(k_m)} - \QI = \left( \hat{\QI}_{m,n}^{(k_m)}- \tilde{\QI}_{m,n}^{(k_m)}\right) +\left( \tilde{\QI}_{m,n}^{(k_m)} - \QI\right) ,
\end{equation}
with 
\begin{equation*}
  \tilde{\QI}_{m,n}^{(k_m)} = \frac{1}{m}\sum_{j=1}^m w^{(k_m)}_j \psi(X'_j).
\end{equation*}
As $\psi$ is globally Lipschitz continuous and does not depend on $\Theta$, we have
\begin{equation*}
    \begin{split}
     \Exp\left[\left( \tilde{\QI}_{m,n}^{(k_m)} - \QI\right)^{2} \right]^{1/2} & \leq L \Exp\left[W_1^{2}\left(\mu_X,\hat{\mu}_{\mathbf{X}'_m}^{(k_m)}\right) \right]^{1/2} \\
     & \leq 2L \left(\Exp\left[W_1^{2}\left(\mu_X,\hat{\mu}_{\mathbf{X}_n} \right)\right] +\Exp \left[W_1^{2}(\hat{\mu}_{\mathbf{X}_n},\hat{\mu}_{\mathbf{X}'_m}^{(k_m)}) \right]\right)^{1/2}
    \end{split} 
\end{equation*} 
by Jensen's inequality, with $L$ the Lipschitz constant of $\psi$.
The second term is bounded from above by $\Exp[W_{2}^{2}(\hat{\mu}_{\mathbf{X}_n},\hat{\mu}_{\mathbf{X}'_m}^{(k_m)})]$, which goes to $0$ by Theorem~\ref{theo:consist}. For the first term, the same arguments as in the proof of Lemma~\ref{lem:W1q} show that
\begin{equation*}
  \Exp\left[W_1^{2}\left(\mu_X,\hat{\mu}_{\mathbf{X}_n} \right)\right] = \grandO\left(\frac{1}{n} + \Exp\left[W_1\left(\mu_X,\hat{\mu}_{\mathbf{X}_n} \right)\right]^2\right).
\end{equation*}
Since $W_1(\mu_X,\hat{\mu}_{\mathbf{X}_n})$ converges to $0$ in probability~\cite{panaretos2019statistical} and the assumption that $\Exp[|X|^2] < +\infty$ ensures that this sequence is uniformly integrable, we deduce that its expectation converges to $0$~\cite[Section 5]{billingsley1999convergence}. Thus, the second part of the right-hand side of~\eqref{eq:decompnoisy} converges to $0$ in $L^2$.

Let us consider the first part in the right-hand side of~\eqref{eq:decompnoisy}. We write the quadratic error
\begin{equation*}
    \begin{split}
   &\Exp\left[ \left| \hat{\QI}_{m,n}^{(k_m)}- \tilde{\QI}_{m,n}^{(k_m)}\right|^{2}\right]\\&= \Exp\left[\left( \frac{1}{m}\sum_{j=1}^{m}w_{j}^{(k_m)}(\psi(X'_{j}) - \phi(f(X'_{j},\Theta_{j})))\right)^{2} \right]  
       \\ & = \Exp\left[\frac{1}{m^{2}}\sum_{j=1}^{m} w_{j}^{(k_m)2}\left(\psi(X'_{j}) - \phi(f(X'_{j},\Theta_{j}))\right)^{2} \right] \\
       & +  \Exp\left[\frac{m-1}{m^{2}}\sum_{j\not=l} w_{j}^{(k_m)} w_{l}^{(k_m)}\left(\psi(X'_{j}) - \phi(f(X'_{j},\Theta_{j}))\right) \left( \psi(X'_{l}) - \phi(f(X'_{l},\Theta_{l}))\right)\right].
    \end{split}
\end{equation*}
Using the fact that $\Exp[w^{(k_m)}_{j}f(X'_{j},\Theta_{j}) |\mathbf{X}_n,\mathbf{X}'_m] = w^{(k_m)}_{j} \psi(X'_{j})$ by definition and the independence of the $\Theta_{j}$, the cross terms vanish. The remaining quadratic term is
\begin{equation}
\begin{split}
&\Exp\left[\frac{1}{m^{2}}\sum_{j=1}^{m} w_{j}^{(k_m)2}\left(\psi(X'_{j}) - \phi(f(X'_{j},\Theta_{j}))\right)^{2} \right]\\
 & =
\frac{1}{m^{2}}\sum_{j=1}^{m}\Exp\left[w^{(k_m)2}_j(\psi(X'_j) - \phi(f(X'_{j},\Theta)))^{2}  \right] \\
&\leq \frac{4}{m^{2}}\sum_{j=1}^{m}\Exp\left[\left(w^{(k_m)}_j\right)^{2}\right]\|\phi\|^{2}_{\infty}.
 \end{split}
 \end{equation}
 We remark that
 \begin{equation*}
  \sum_{j=1}^m \left(w^{(k_m)}_j\right)^2 = \frac{m^2}{n^2 k_m^2}\sum_{i_1,i_2=1}^n\sum_{l_1,l_2=1}^{k_m} \ind{j^{(l_1)}_{i_1}=j^{(l_2)}_{i_2}}
\end{equation*}
and that for some fixed $i_1$,$i_2$ and $l_1$, there exists exactly one $l_2 \in \llbracket 1,m \rrbracket$ such that $j^{(l_1)}_{i_1}=j^{(l_2)}_{i_2}$ as $(j_{i_2}^{(l)})_{1\leq l \leq m}$ is a permutation of $\llbracket 1,m \rrbracket$. Therefore, there exists at most one $l_2 \in \llbracket 1,k_{m} \rrbracket$ verifying this property and, consequently,
\begin{equation*}
 \sum_{j=1}^m \left(w^{(k_m)}_j\right)^2  \leq \frac{m^2}{n^2 k_m^2}\sum_{i_1,i_2=1}^n\sum_{l_1=1}^{k_m} 1= \frac{m^{2}}{k_m}.
\end{equation*}
We can then bound the second term by
\begin{equation*}
    \Exp\left[\left( \frac{1}{m}\sum_{j=1}^{m}w_{j}^{(k_m)}(\psi(X'_{j}) - \phi( f(X'_{j},\Theta_{j})) ) \right)^{2} \right]^{1/2}  \leq \frac{2}{k_m^{1/2}} \|\phi\|_{\infty},
\end{equation*} 
which converges to $0$ when $k_m$ goes to infinity.
\end{proof}

In order to complement Proposition~\ref{prop:consisnoisycase} with a rate of convergence, we restart from the decomposition~\eqref{eq:decompnoisy}. Under the additional assumptions of Corollary~\ref{corr:ratesk} with $q=2$, the same arguments as in the proof of Proposition~\ref{prop:ratenoiseless} yield
\begin{equation*}
 \Exp\left[\left( \tilde{\QI}_{m,n}^{(k_m)} - \QI\right)^{2} \right]^{1/2} = \grandO\left(\frac{1}{\sqrt{n}} + \tau_{1,d}(n)\right) + \grandO\left( \left(\frac{k_m}{m} \right)^{1/d}\right),
\end{equation*}
while we still have
\begin{equation*}
  \Exp\left[ \left| \hat{\QI}_{m,n}^{(k_m)}- \tilde{\QI}_{m,n}^{(k_m)}\right|^{2}\right]^{1/2}  = \grandO\left(\frac{1}{\sqrt{k_m}}\right)
\end{equation*}
from the proof of Proposition~\ref{prop:consisnoisycase}. As a consequence,
\begin{equation*}
  \Exp\left[\left|\QI - \hat{\QI}_{m,n}^{(k_m)}\right|^{2}\right]^{1/2} = \grandO\left(\frac{1}{\sqrt{n}} + \tau_{1,d}(n)\right) + \grandO\left( \left(\frac{k_m}{m} \right)^{1/d} + \frac{1}{\sqrt{k_m}}\right).
\end{equation*}
Optimizing in $k_m$, we get the following statement.

\begin{prop}[Rates of convergence in the noisy case]
\label{prop:ratenoisy}
  Under the assumptions of Proposition~\ref{prop:consisnoisycase} and Corollary~\ref{corr:ratesk} with $q=2$, we have
\begin{equation*}
  \Exp\left[\left|\QI - \hat{\QI}_{m,n}^{(k_m)}\right|^{2}\right]^{1/2} = \grandO\left(\frac{1}{\sqrt{n}} + \tau_{1,d}(n)\right) + \grandO\left(\frac{1}{m^{1/(d+2)}}\right)
\end{equation*}
for $k_m \sim m^{2/(d+2)}$.
\end{prop}
The loss of convergence order with respect to Proposition~\ref{prop:ratenoiseless} is similar to the NNR context,  in which it deteriorates from the rate  $1/d$ in the noiseless case to the rate of $1/(d+2)$ in the noisy case \cite[Section 14.6 and Section 15.3]{biau2015lectures}.

\section{Applications and numerical illustration}\label{sec:appli}

We present a reformulation of our results in a standard framework for $k$-NN regression in Subsection~\ref{ss:generr}, and then provide a detailed account of the original motivation of this work by decomposition-based UQ in Subsection~\ref{sec:appUQ}. Last, numerical illustrations of our main results in a simple setting are reported in Subsection~\ref{sec:num}; we refer to~\cite[Chapter 11]{touboul2021} for an application in an industrial context.

\subsection{Generalization error of \texorpdfstring{$k$}{k}-NN regression under covariate shift}\label{ss:generr}

In this subsection, we address the $k$-NN regression problem under covariate shift from the following more standard point of view: the quantity of interest is directly the regression function
\begin{equation*}
  r(x) := \Exp\left[Y|X=x\right] = \Exp\left[f(x,\Theta)\right],
\end{equation*}
and the $k$-NN estimator of $r(x)$ is defined from the training set by
\begin{equation*}
  \hat{r}_m^{(k)}(x) = \frac{1}{k}\sum_{l=1}^k Y'_{j^{(l)}(x)},
\end{equation*}
where $j^{(l)}(x)$ denotes the (smallest) index $j$ such that $X'_j = \NN^{(l)}_{\mathbf{X}'_m}(x)$. We are no longer interested in some quantity $\Exp[\phi(Y)]$ but rather in the \emph{($L^2$) generalization error} under covariate shift
\begin{equation*}
  \Exp\left[\left|r(X)-\hat{r}^{(k)}_m(X)\right|^2\right], \qquad X \sim \mu_X.
\end{equation*}
For the sake of simplicity we assume that $Y$, $r(X)$, etc. take scalar values.

\begin{theo}[$L^2$ generalization error of the $k$-NN regression under covariate shift]
  Let $\mu_X$ and $\mu_{X'}$ verify the assumptions of Theorem~\ref{theo:rates} for $q=2$. Assume in addition that $f$ is Lipschitz continuous in $x$, uniformly in $\Theta$, and that $\Var(f(x,\Theta)) \leq \sigma^2<+\infty$ for all $x \in \supp(\mu_{X'})$.
  
  When the sequence $(k_m)_{m \geq 1}$ satisfies the assumptions of Corollary~\ref{corr:ratesk}, with $k_m \sim m^{2/(d+2)}$, there exists $C \geq 0$ such that
  \begin{equation*}
  \underset{m \rightarrow + \infty}{\limsup}\  m^{1/(2+d)}\Exp\left[\left(r(X)-\hat{r}^{(k_m)}_m(X)\right)^2\right]^{1/2}  \leq \sigma +C \Exp\left[\frac{1}{p_{X'}(X)^{2/d}}\right]^{1/2}.
  \end{equation*}
  \end{theo}
  We retrieve essentially the same orders of convergence as in the case without covariate shift. The quantity $\Exp[1/p_{X'}(X)^{2/d}]^{1/2}$ seems to be the relevant bound of the loss due to the use of $\mu_{X'}$ instead of $\mu_X$ and we expect that the greater this quantity is, the slower the convergence will be.
\begin{proof}
  The proof is an adaptation of  \cite[Theorem 14.5]{biau2015lectures}, using elements of the proofs of Theorem~\ref{theo:rates} and  Corollary~\ref{corr:ratesk}.  We can decompose the $L^2$ error
  \begin{align*}
    &\Exp\left[\left( r(X) - \hat{r}^{(k_m)}_{m}(X) \right)^2\right]^{1/2}\leq    \Exp\left[\left( r(X) - \tilde{r}^{(k_m)}_{m}(X) \right)^2\right]^{1/2}+   \Exp\left[\left( \tilde{r}^{(k_m)}_m(X) - \hat{r}^{(k_m)}_{m}(X) \right)^2\right]^{1/2}
  \end{align*}
  with 
  \begin{equation*}
    \tilde{r}^{(k_m)}_m(x) =  \frac{1}{k_m} \sum_{l=1}^{k_m}\Exp\left[f(\NN^{(l)}_{\mathbf{X}'_m}(x), \Theta)\right].
  \end{equation*}
  By Jensen's inequality, the first term can be bounded by
  \begin{equation*}
    \Exp\left[\left( r(X) - \tilde{r}^{(k_m)}_m(X) \right)^2\right] \leq L^2\left(\frac{1}{k_m}\sum_{l=1}^{k_m}\Exp\left[\left( X - \NN_{\mathbf{X}'_m}^{(l)}(X) \right)^2\right]\right),
  \end{equation*}
  where $L$ is the Lipschitz constant of $f$, and then following the proof of Corollary~\ref{corr:ratesk}, we get
  \begin{equation*}
    \underset{m \rightarrow + \infty}{\limsup} \left(\frac{m}{k_m} \right)^{1/d}\Exp\left[\left( r(X) - \tilde{r}^{(k_m)}_m(X) \right)^2\right]^{1/2}  \leq C\Exp\left[\frac{1}{p_{X'}(X)^{2/d}}\right]^{1/2},
  \end{equation*}
  with $C := L (c_{d,2}\Gamma(1+2/d)/v_d^{2/d})^{1/2}$. The second term is bounded by
  \begin{equation*}
  \begin{split}
    &\Exp\left[\left( \tilde{r}^{(k_m)}_m(X) - \hat{r}^{(k_m)}_m(X) \right)^{2}\right]^{1/2}\\
    & = \frac{1}{k_m}\Exp\left[\sum_{l=1}^{k_m}\left(f(\NN^{(l)}_{\mathbf{X}'_m}(X), \Theta_{j^{(l)}(X)}) - \Exp[f(\NN^{(l)}_{\mathbf{X}'_m}(X), \Theta)| X]\right)^{2}\right]^{1/2} \\
    & \leq \frac{1}{k_m^{1/2}} \sigma.
  \end{split}
  \end{equation*}
  The optimal rate is $k_m \sim m^{2/(2+d)}$, leading to
  \begin{equation*}
    \underset{m \rightarrow + \infty}{\limsup}\ m^{1/(2+d)}\Exp\left[\left( r(X) - \hat{r}^{(k_m)}_m(X) \right)^2\right]^{1/2}  \leq \sigma + C\Exp\left[\frac{1}{p_{X'}(X)^{2/d}}\right]^{1/2},
  \end{equation*}
  which completes the proof.
\end{proof}
\subsection{Application to decomposition-based UQ} \label{sec:appUQ} In the UQ context, the relation~\eqref{eq:Y} represents a computer simulation~\cite{fang2005design,de2008uncertainty}: the random variable $X$ is the \emph{input} of the simulation, the random variable $\Theta$ describes the set of its \emph{parameters}, the function $f$ is the \emph{numerical model} and the random variable $Y$ is the \emph{output} of the simulation. The function $\phi$ involved in the definition~\eqref{eq:QI} of the quantity of interest $\QI$ is the \emph{observable}.

The fact that we assume that both $X$ and $\Theta$ may be random, but with distinct sources of uncertainty (which is modeled by their statistical independence), comes from the study of uncertainty propagation in complex networks of numerical models~\cite{AllWil10, amaral2014decomposition, MarGarPer19, SanLeMCon19, MinGui}. In this context, several computer codes, representing various disciplines, are connected with each other by the fact that the outputs of certain codes are taken as inputs of other codes. Then $f$ represents one discipline, with `internal' uncertain parameters $\Theta$ whose law is known by the agent in charge of the simulation, and `external' uncertain parameters $X$ which are the output of possibly several upstream numerical simulations. Independently from our complex system context, assuming that the internal parameter $\Theta$ may be random is a standard practice to take aleatoric or epistemic uncertainty into account~\cite{de2008uncertainty,ghanem2017handbook}. 

If $X$ is deterministic then the computation of a global quantity of interest can be treated by the so-called Collaborative Optimization methods~\cite{braun1995development,yao2011review} in Multidisciplinary Analysis and Optimization. However, if $X$ is random, a direct Monte Carlo evaluation of $\QI$ is often impossible to implement in practice. Indeed, if the number of interacting disciplines is large and each code evaluation is costly, then one cannot wait for a sample $X_1, \ldots, X_n$ to be generated by the upstream simulations before starting running one's own simulation. Therefore, decomposition-based UQ methods have been introduced in the literature in order to allow disciplines to run their numerical simulations independently. Basically, these methods work in two phases. In an \emph{offline} phase, each discipline generates it own \emph{synthetic} sample $X'_1, \ldots, X'_m$ according to some user-chosen probability measure $\mu_{X'}$ on $\R^d$ (or possibly other designs of experiment). The numerical model $f$ is then evaluated on the sample $(X'_1, \Theta_1), \ldots, (X'_m,\Theta_m)$ to obtain a corresponding set of realizations $Y'_1, \ldots, Y'_m$. Once actual realizations $X_1, \ldots, X_n$ become available in a subsequent \emph{online} phase, they have to be used in combination with the synthetic sample to construct an estimator of $\QI$, but evaluations of the numerical model $f$ are no longer allowed. 

We refer to~\cite{amaral2012decomposition,amaral2014decomposition,amaral2017optimal} for examples and background on these methods. The $k$-NN reweighting scheme introduced in the present article is yet another possible approach to this problem. More general nonparametric regression methods, such as Nadaraya--Watson estimators, may also be considered. A more systematic study of such approaches, based on linear reweighting, as well as their generalization to the estimation of quantities of interest defined on a \emph{graph} of numerical models, may be found in~\cite[Chapter 11]{touboul2021} and will be the object of a future publication.

\begin{rk}[Stochastic simulators]\label{rk:stochsim}
  Our framework is also suited to the situation where $\Theta$ does not represent well-identified \emph{parameters}, but must rather be interpreted as the inherent randomness of the numerical model. In the UQ literature, such models are called \emph{stochastic simulator} (see for instance~\cite{ZhuSud21} and the references therein) and their emulation is closely related with the regression problem addressed in this article, interpreting $\Theta$ as a \emph{noise} term.
\end{rk}

\begin{rk}[A simple example with low-dimensionally supported data]\label{rk:abscont}
  In the multidisciplinary context introduced above, consider the simple setting in which a random variable $X^1$ is taken as an input by two distinct disciplines, represented by two numerical models $f^1$ and $f^2$. Assume in addition that the output $Y^1 = f^1(X^1)$ of the former is taken as an input by the latter, so that $Y^2 = f^2(X^1,Y^1)$. In the decomposition-based approach, the second discipline has to design a synthetic sample $(X'^1_j,Y'^1_j)_{j \in \llbracket 1, m\rrbracket}$ without the actual knowledge of $f^1$. Therefore it is unlikely that this sample be absolutely continuous with respect to the true law of $(X^1,Y^1)$, the support of which lies in the manifold $\{(x,f^1(x))\}$.
\end{rk}

\subsection{Numerical illustrations}\label{sec:num}

This subsection investigates numerically the influence of the choice of the synthetic distribution $\mu_{X'}$ on the quality of the respective approximations of $\mu_X$ by $\hat{\mu}^{(1)}_{\mathbf{X}'_m}$, and of $\QI$ by $\hat{\QI}^{(k_m)}_{m,n}$. 

\subsubsection{Influence of \texorpdfstring{$\mu_{X'}$}{} on the convergence of \texorpdfstring{$\hat{\mu}^{(1)}_{\mathbf{X}'_m}$}{}} \label{sec:nummes} We investigate how the relationship between $\mu_X$ and $\mu_{X'}$ impacts the convergence of $\hat{\mu}^{(1)}_{\mathbf{X}'_{m}}$ presented in Subsection~\ref{sec:discconvtheo}. In this numerical experiment, we set the dimension $d=2$, choose
\begin{equation*}
 X = (U,U),\quad U \sim \mathcal{U}\left([0,1] \right),
\end{equation*}
  and 
  \begin{equation*}
  X'
\sim \mathcal{N}\left(\begin{pmatrix} \mu \\ \mu \end{pmatrix}, \sigma^{2}\begin{pmatrix}
1 & s_{\mathrm{corr}} \\
s_{\mathrm{corr}}  & 1 \\                                                                                                                                                                                                                                                                                                                \end{pmatrix}\right),
\end{equation*}
with $\mu = 0.5$, $\sigma = 0.3$ and various $s_{\mathrm{corr}}$ in $  (-1,1)$. Intuitively, the closer $s_{\mathrm{corr}}$ is to $1$, the closer  $\mu_{X'}$ is to $\mu_{X}$, as illustrated in Figure~\ref{fig:compdist}.

\begin{figure}[!hbt]
  \begin{minipage}{0.33\textwidth}
    \centering
    \includegraphics[width=\textwidth]{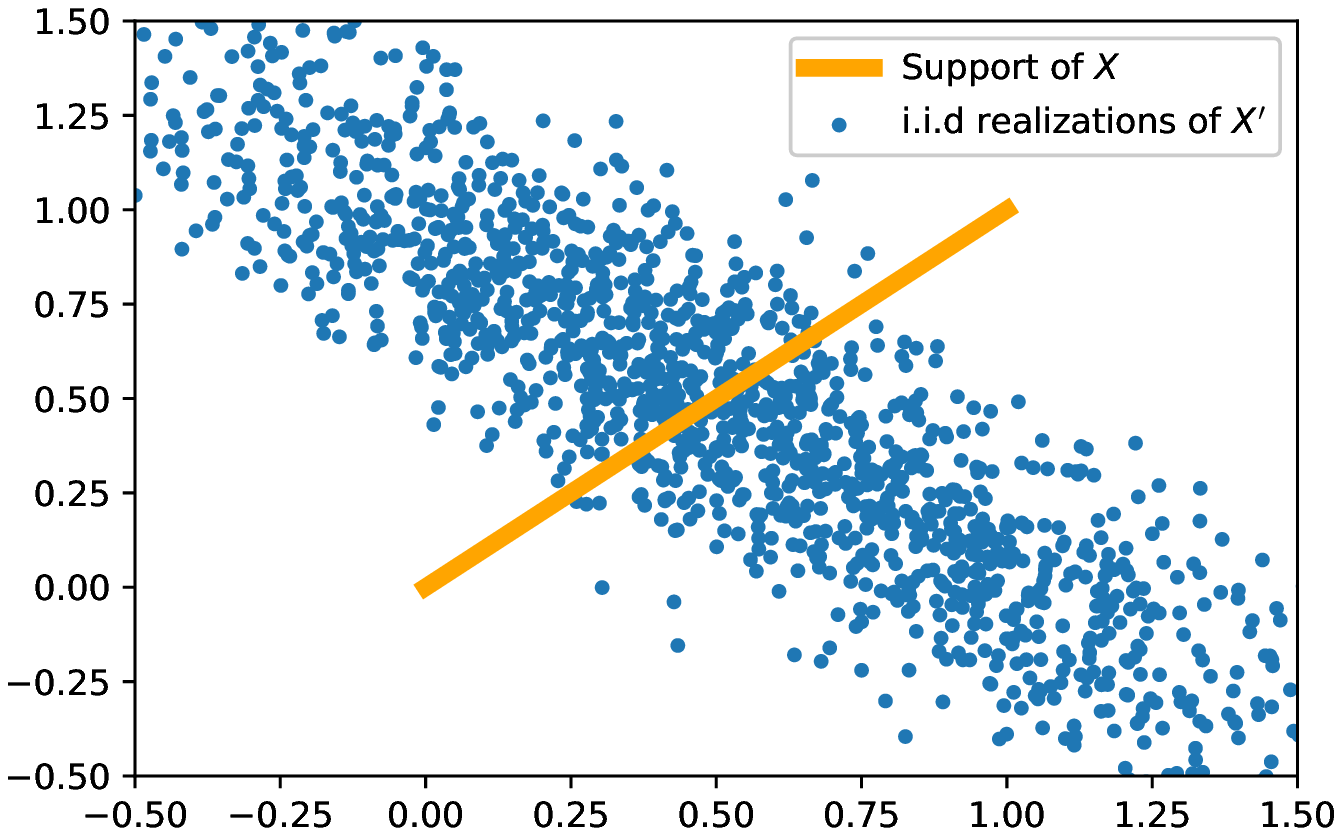}\\
    (a)~$s_{\mathrm{corr}} = -0.9$
  \end{minipage}\begin{minipage}{0.33\textwidth}
    \centering
    \includegraphics[width=\textwidth]{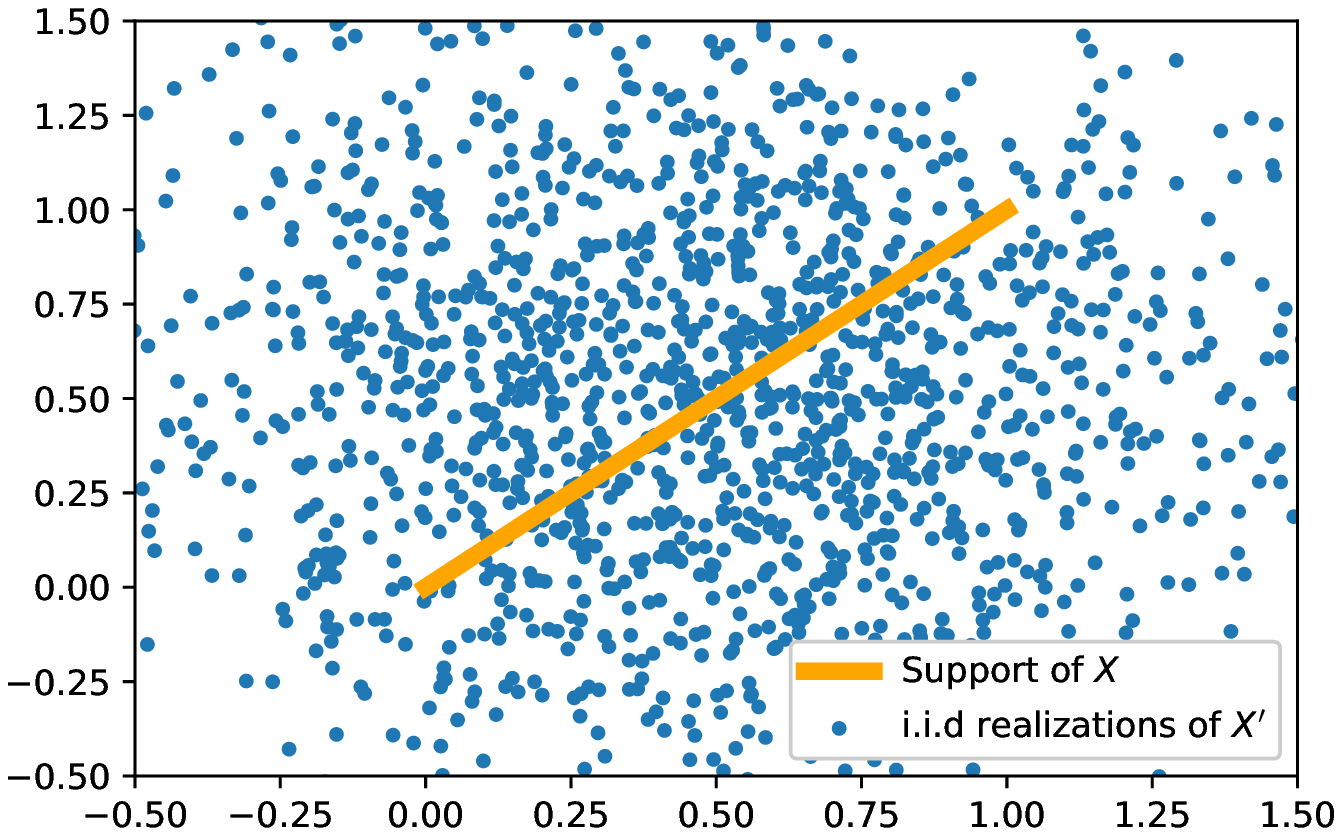}\\
    (b)~$s_{\mathrm{corr}} = 0$
  \end{minipage}\begin{minipage}{0.33\textwidth}
    \centering
    \includegraphics[width=\textwidth]{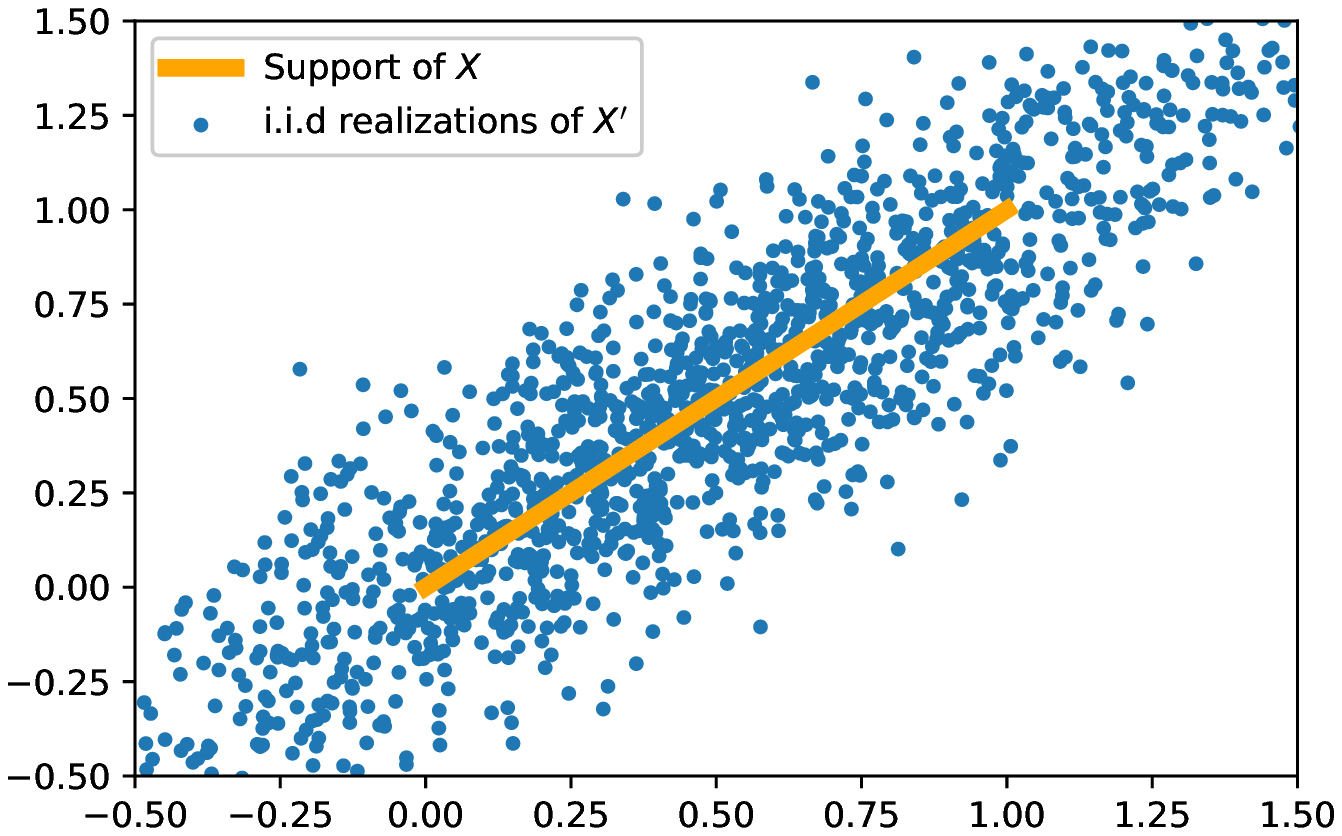}\\
    (c)~$s_{\mathrm{corr}} = 0.9$
  \end{minipage}
  \caption{Plot of the support of $X$ and $1500$ i.i.d realizations of $X'$ for different values of $s_{\mathrm{corr}}$.}\label{fig:compdist}
\end{figure}

As a first `purely visual' indication of the quality of the approximation of $\mu_X$ by $\hat{\mu}^{(1)}_{\mathbf{X}'_m}$, we plot on Figure~\ref{fig:kde} the trace of a kernel smoothing of $\hat{\mu}^{(1)}_{\mathbf{X}'_{m}}$ on the segment $\{(u,u), u \in [0,1]\}$. We can see that the greater $s_{\mathrm{corr}}$ is, the better the reconstruction looks like. From a more quantitative point of view, this observation is confirmed in Figure~\ref{fig:evolutionVal}, where we plot the evolution of $\Exp[W_{2}^{2}(\hat{\mu}_{\mathbf{X}_n}, \hat{\mu}^{(1)}_{\mathbf{X}'_{m}} ) ]$ as a function of $m$. We can see that although this quantity converges at the theoretical rate $m^{-1}$, the multiplicative constant decreases with $s_{\mathrm{corr}}$.
\begin{figure}[!hbt]  
\centering  
        \includegraphics[width=.8\textwidth]{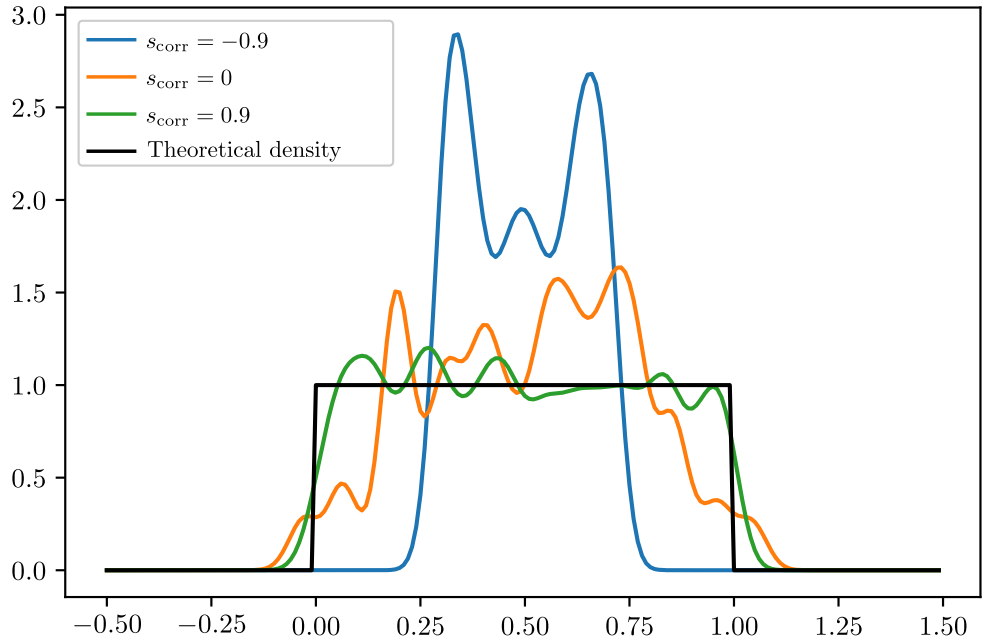}
        \caption{Trace of a kernel smoothing of $\hat{\mu}^{(1)}_{\mathbf{X}'_{m}}$ on the segment $\{(u,u), u \in [0,1]\}$ for different values of $s_{\mathrm{corr}}$.}\label{fig:kde}
\end{figure}
\begin{figure}[!hbt]
\centering  
        \includegraphics[width=.8\textwidth]{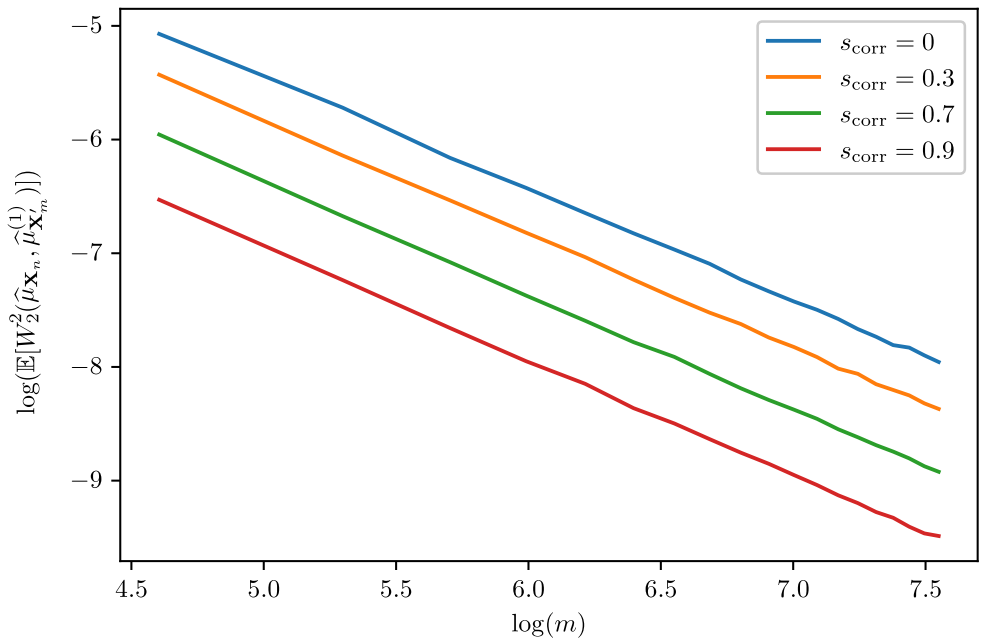}
\caption{Evolution of $\Exp [W_{2}^{2}(\hat{\mu}_{\mathbf{X}_n}, \hat{\mu}^{(1)}_{\mathbf{X}'_{m}} ) ]$ with respect to $m$, for various values of $s_{\mathrm{corr}}$. For each value of $m$ and $s_{\mathrm{corr}}$, the quantity $W_{2}^{2}(\hat{\mu}_{\mathbf{X}_n}, \hat{\mu}^{(1)}_{\mathbf{X}'_{m}} )$ is computed (with $n=100$) thanks to the formula~\eqref{eq:upbound1nn}, and its expectation is estimated through Monte Carlo average over $700$ replications. The theoretical rate of $m^{-q/d} = m^{-1}$ is attained for each experiment.}\label{fig:evolutionVal}
\end{figure}        
  
\subsubsection{Influence of \texorpdfstring{$\mu_{X'}$}{} on the convergence of \texorpdfstring{$\hat{\QI}_{m,n}^{(k_m)}$}{}}
We now concentrate on the impact on the efficiency of $\hat{\QI}_{m,n}^{(k_m)}$. We keep the framework of the previous paragraph, and we try to estimate the quantity of interest
 \begin{equation*}
  \QI =\Exp[\phi(f(X,\Theta))], \quad f((x_1,x_2),\theta) = \sin(2\pi x_1)\sin(2\pi x_2)(1+\theta)
 \end{equation*}
with $\Theta \sim \mathcal{U}(\left[-1,1 \right])$ and $\phi(y) = y$. The $L^{2}$ error  
\begin{equation*}  
\Exp\left[ \left| \hat{\QI}_{m,n}^{(k_m)}  - \QI \right|^{2} \right]^{1/2} = \Exp\left[ \left| \hat{\QI}_{m,n}^{(k_m)}  - 0.5\right|^{2} \right]^{1/2}
\end{equation*}
is computed by Monte Carlo estimation. As highlighted in Figure~\ref{fig:L2err}, the closeness of $\mu_X$ to $\mu_{X'}$ is an important factor for the efficiency of the estimator.   
 
\begin{figure}[!hbt] 
\centering  
 \includegraphics[width=.8\textwidth]{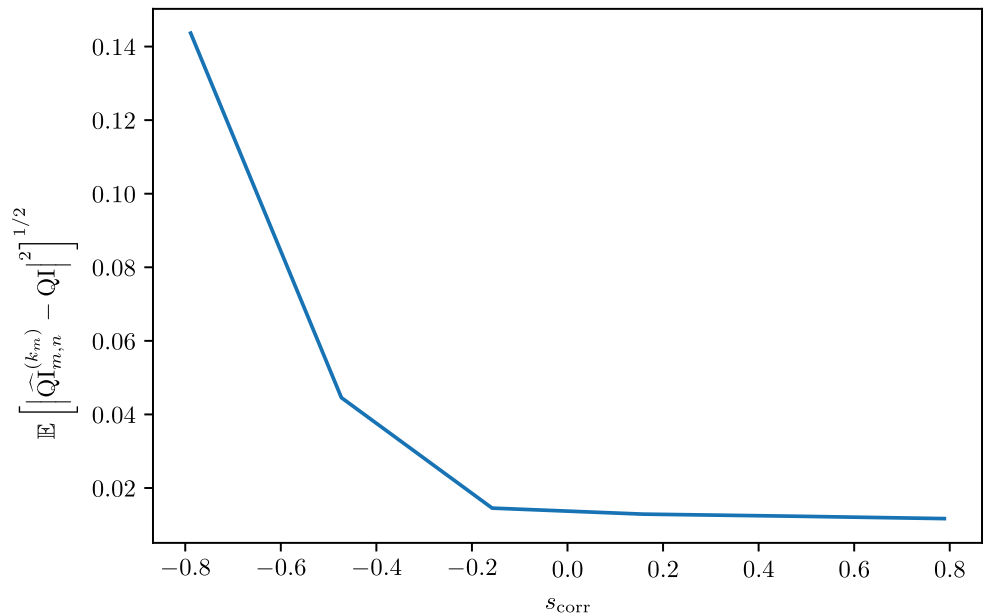} 
 \caption{Estimation of the $L^{2}$ error with respect to $s_{\mathrm{corr}}$ for $n=m = 900$ and $k_{m}=4$, averaged on $2000$ replications.}\label{fig:L2err}
\end{figure}

\subsection*{Acknowledgements} This work was motivated by a collaboration with P. Benjamin, F. Mangeant and M. Yagoubi. We also benefited from fruitful discussions with G. Biau and A. Guyader. Last, we thank two anonymous referees for their careful reading of the article, and their numerous suggestions which allowed to greatly improve the presentation of this work.
 


\end{document}